\documentclass[12pt]{amsart}

\usepackage{amsfonts}
\usepackage{enumerate}
\usepackage{amssymb}
\usepackage{xfrac}	
\usepackage{slashed} 
\usepackage{color}
\usepackage{amsmath}
\usepackage[bookmarksopen,bookmarksdepth=2]{hyperref}
\providecommand{\noopsort}[1]{} 
\usepackage{mathtools}

\usepackage{amsthm}
\newtheorem{thm}{Theorem}[section]
\newtheorem{lem}[thm]{Lemma}
\newtheorem{prop}[thm]{Proposition}

\newtheorem{defn}[thm]{Definition}

\newtheorem{remark}[thm]{Remark}

\def\la{{\lambda}}

\newcommand{\N}{\mathbb N}
\newcommand{\Z}{\mathbb Z}

\newcommand{\R}{\mathbb R}
\newcommand{\C}{\mathbb C}

\renewcommand{\S}{\mathcal{S}}

\renewcommand{\H}{\mathcal{H}}
\renewcommand{\L}{\mathcal{L}}

\newcommand{\mB}{\mathcal{B}}

\newcommand{\mW}{\mathfrak{W}}

\renewcommand{\O}{O}

\def\Re{{\mathrm{Re}\,}}
\def\Im{{\mathrm{Im}\,}}

\newcommand{\supnorm}[1]{\norm{ #1 }_\infty}
\newcommand{\absnorm}[1]{\norm{ #1 }_1}

\newcommand{\norm}[1]{\left\| #1 \right\|}
\newcommand{\nrm}[2]{\left\|#1\right\|_{#2}}
\newcommand{\tnrm}[2]{\|#1\|_{#2}}

\newcommand{\vect}[2]{\begin{pmatrix}#1\\#2\end{pmatrix}}

\newcommand{\dom}{\textnormal{dom}}

\newcommand{\supp}{\operatorname{supp}}

\newcommand{\Tr}{\operatorname{Tr}}

\begin{document}

\date{\today}

\title{Spectral shift for relative Schatten class perturbations}

\author[van Nuland]{Teun D.H. van Nuland$^{*}$}
\address{T.v.N., Institute for Mathematics, Astrophysics and Particle Physics, Radboud University, Heyendaalseweg 135, 6525 AJ Nijmegen, The Netherlands}
\email{t.vannuland@math.ru.nl}

\thanks{\footnotesize $^{*}$Research supported by NWO Physics Projectruimte (680-91-101)}

\author[Skripka]{Anna Skripka$^{**}$}
\address{A.S., Department of Mathematics and Statistics, University of New Mexico, 311 Terrace Street NE, Albuquerque, NM  87106, USA}
\email{skripka@math.unm.edu}

\thanks{\footnotesize $^{**}$Research supported in part by NSF grant DMS-1554456}



\begin{abstract}
We affirmatively settle the question on existence of a real-valued higher order spectral shift function for a pair of self-adjoint operators $H$ and $V$ such that $V$ is bounded and $V(H-iI)^{-1}$ belongs to a Schatten-von Neumann ideal $\S^n$ of compact operators in a separable Hilbert space. We also show that the function satisfies the same trace formula as in the known case of $V\in\S^n$ and that it is unique up to a polynomial summand of order $n-1$. Our result significantly advances earlier partial results where counterparts of the spectral shift function for noncompact perturbations lacked real-valuedness and aforementioned uniqueness as well as appeared in more complicated trace formulas for much more restrictive sets of functions.
Our result applies to models arising in noncommutative geometry and mathematical physics.
\end{abstract}

\maketitle

\section{Introduction}
\label{sec1}

The spectral shift function originates from the foundational work of M.G.~Krein \cite{Krein53} which followed I.M.~Lifshits's physics research summarised in \cite{Lifshits}.
It is a central object in perturbation theory that allows to approximate a perturbed operator function by the unperturbed one while controlling noncommutativity in the remainder.
In 1984, Koplienko \cite{Koplienko84} suggested an interesting and useful generalization by considering higher order Taylor remainders and conjecturing existence of higher order spectral shift functions. Many partial results were obtained in that direction, but they were confined to either lower order approximations, weakened trace functionals and representations, or compact perturbations.
This paper closes a gap between theory and applications, where perturbations are often noncompact, by proving existence of a higher order spectral shift function under a general condition on a weighted resolvent of the initial operator and obtaining bounds and properties stricter than previously known.

Our prime result is that, given a self-adjoint operator $H$ densely defined in a separable Hilbert space $\H$ and a bounded self-adjoint operator $V$ on $\H$ satisfying
\begin{align}
\label{vresinsn}
V(H-iI)^{-1}\in\S^n,
\end{align}
there exists a real-valued spectral shift function $\eta_n=\eta_{n,H,V}$ of order $n$. Namely, the trace formula
\begin{align}
\label{trvinsn}
\Tr\left(f(H+V)-\sum_{k=0}^{n-1}\frac{1}{k!}\frac{d^k}{dt^k} f(H+tV)\big|_{t=0}\right)
=\int_\R f^{(n)}(x)\,\eta_n(x)\,dx
\end{align}
holds for a wide class of functions $f$ and the function $\eta_n$ satisfies suitable uniqueness and summability properties and bounds detailed below.
The {\it relative Schatten class condition} \eqref{vresinsn} applies, in particular, to
\begin{enumerate}[(I)]
	\item\label{vinsn} $V\in\S^n;$
	\item\label{rinsn} $(H-iI)^{-1}\in\S^n$;
	\item\label{inner perturbations} inner fluctuations of $H=D$ in a regular locally compact spectral triple $(\mathcal{A},\H,D)$ (see Section \ref{sec4a});
	\item\label{diff operators} differential operators on manifolds perturbed by multiplication operators (see Section \ref{sec4b}).
\end{enumerate}

To prove our main result, we develop new techniques, which were also applied in the subsequent work \cite{vNvS21a} in the setting \eqref{rinsn} to resolve analytical issues occurring in the study of the spectral action in noncommutative geometry. The latter application suggests that our techniques can be used to substantially generalize \cite{vNvS21a,vNvS21b} as well as can be useful in other problems of noncommutative geometry.

\medskip
\paragraph{\bf New and prior results.}

Under the assumption \eqref{vinsn}, the problem on existence of higher order spectral shift functions was resolved in \cite{PSS13}.
More precisely, \eqref{trvinsn} was established in \cite{Krein53}, \cite{Koplienko84}, \cite{PSS13} for $n=1$, $n=2$, $n\ge 3$, respectively, for important test functions $f$ (see, e.g., \cite[Section 5.5]{ST19} for details), where the function $\eta_n=\eta_{n,H,V}$ is unique, real-valued, and satisfies the bound
\begin{align*}
\|\eta_n\|_1\le c_n\|V\|_n^n.
\end{align*}

Taylor approximations and respective trace formulas were also derived in the study of the spectral action functional $\Tr(f(H))$ occuring in noncommutative geometry \cite{CC97} for operators $H$ with compact resolvent $(H-iI)^{-1}$. The case of \eqref{rinsn} and functions $f$ in the form $f(x)=g(x^2)$, where $g$ is the Laplace transform of a regular Borel measure, was handled in \cite{WvS}. The case of compact $(H-iI)^{-1}$ and
$f\in C_c^{n+1}(\R)$ was handled in \cite{S14,S18}. In particular, the existence of a locally integrable spectral shift function was established in \cite{S18}.

In our main result, Theorem \ref{rsmain}, given $n\in\N$ and $H,V$ satisfying \eqref{vresinsn}, we establish the existence of a real-valued function $\eta_n=\eta_{n,H,V}$ such that $\eta_n\in L^1\big(\R,\tfrac{dx}{(1+|x|)^{n+\epsilon}}\big)$ for every $\epsilon>0$ and such that \eqref{trvinsn} holds for every $f\in\mW_n$, where the class $\mW_n$ is given by Definition \ref{fracw}. In particular, $\mW_n$ includes all $(n+1)$-times continuously differentiable functions whose derivatives decay at infinity at the rate $f^{(k)}(x)=\O\left(|x|^{-k-\alpha}\right)$, $k=0,\ldots,n+1$, for some $\alpha>\frac12$ (see Proposition \ref{inclusions}\eqref{inclusionsi}). The weighted $L^1$-norm of the spectral shift function $\eta_n$ admits the bound
\begin{align*}
\int_\R|\eta_n(x)|\,\frac{dx}{(1+|x|)^{n+\epsilon}}
\le c_n(1+\epsilon^{-1})(1+\norm{V})\|V(H-iI)^{-1}\|_n^n
\end{align*}
for every $\epsilon>0$. Moreover, the locally integrable spectral shift function $\eta_n$ is unique up to a polynomial summand of degree at most $n-1$.

Below we briefly summarize advantages of our main result in comparison to most relevant prior results.
Other results on approximation of operator functions and omitted details can be found in \cite[Chapter 5]{ST19} and references cited therein.

The existence of a real-valued function $\eta_1\in L^1\big(\R,\tfrac{dx}{1+x^2}\big)$ satisfying the trace formula \eqref{trvinsn} with $n=1$ for bounded rational functions was established in \cite[Theorem~3]{Krein62} (see also \cite[p.~48, Corollary 0.9.5]{Yafaev10}). The formula \eqref{trvinsn} was extended to twice-differentiable $f$ with bounded $f',f''$ such that
\begin{align}
\label{yafaevcond}
\frac{d^k}{dx^k}(f(x)-c_f x^{-1})
=\O(|x|^{-k-1-\epsilon})\quad{\rm as }\,\;|x|\rightarrow\infty,\quad k=0,1,2,\quad \epsilon>0,
\end{align}
where $c_f$ is a constant, in \cite[p.~47, Theorem 0.9.4]{Yafaev10}. It was shown in \cite[Section 8.8 (3)]{YGT} that $\eta_1\in L^1\big(\R,\frac{dx}{(1+|x|)^{1+\epsilon}}\big)$ for $\epsilon>0$.
The respective function $\eta_1$ was determined by \eqref{trvinsn} uniquely up to a constant summand.
We prove that \eqref{trvinsn} with $n=1$ holds for all $\mathfrak{W}_1$, which contains all functions satisfying \eqref{yafaevcond} (see Proposition \ref{inclusions}\eqref{inclusionsi}) as well as functions not included in \eqref{yafaevcond} (see, e.g., Remark \ref{mncontains}).

In \cite[Corollary 3.7]{Neidhardt88}, the trace formula \eqref{trvinsn} with $n=2$ and real-valued $\eta_2\in L^1\big(\R,\tfrac{dx}{(1+x^2)^2}\big)$ was proved for a set of functions including Schwartz functions
along with ${\rm span}\,\{(z-\cdot)^{-k}:\, \Im(z)\neq 0,\, k\in\N,\, k\ge 2\}$.
The respective $\eta_2\in L^1\big(\R,\tfrac{dx}{(1+x^2)^2}\big)$ was determined by \eqref{trvinsn} uniquely up to a linear summand.
We prove that \eqref{trvinsn} with $n=2$ holds for all $f\in\mW_2$, which contains the functions $(z-\cdot)^{-1}$, $\Im(z)\neq 0$ not included in \cite[Corollary 3.7]{Neidhardt88} and the Schwartz functions included in \cite[Corollary 3.7]{Neidhardt88}, and that
$\eta_2$ is integrable with a significantly smaller weight, namely, $\eta_2\in L^1\big(\R,\frac{dx}{(1+|x|)^{2+\epsilon}}\big)$ for $\epsilon>0$.

Let $n\ge 2$. The existence of a complex-valued $\tilde\eta_n\in L^1\big(\R,\tfrac{dx}{(1+x^2)^{n/2}}\big)$ satisfying the trace formula
\begin{align}
\label{trcs18}
\Tr\Big(f(H+V)-\sum_{k=0}^{n-1}\frac{1}{k!}\frac{d^k}{dt^k}f(H+tV)|_{t=0}\Big)
=\int_\R\frac{d^{n-1}}{dx^{n-1}}\big((x-i)^{2n}f'(x)\big)\tilde\eta_n(x)\,dx
\end{align}
for a set of functions $f$ including ${\rm span}\,\{(z-\,\cdot)^{-k},\; \Im(z)>0,\; k\in\N,\; k\ge 2n\}$ was established in \cite[Theorem 4.6]{CS18} (see also \cite[Remark 4.8(ii)]{CS18}). The weighted $L^1$-norm of $\tilde\eta_n$ satisfies the bound
\begin{align*}
\int_\R|\tilde\eta_n(x)|\,\frac{dx}{(1+x^2)^{\frac{n}{2}}}
\le c_n(1+\|V\|)^{n-1}\|V(H-iI)^{-1}\|_n^n.
\end{align*}
As distinct from the aforementioned result of \cite{CS18} for $n\ge 2$, the function $\eta_n$ in our main result is real-valued and satisfies the  simpler trace formula \eqref{trvinsn} for the larger class $\mW_n$ of functions $f$ described in terms of familiar function classes. Moreover, the set of functions $\mW_n$ is large enough to ensure the uniqueness of $\eta_n$ up to a polynomial term of degree at most $n-1$.
\medskip

Other assumptions on $H$ and $V$, each having its merits and limitations, were also considered in the literature. For instance, the existence of a nonnegative function $\eta_2=\eta_{2,H,V}\in L^1\big(\R,\tfrac{dx}{(1+x^2)^\gamma}\big)$, $\gamma>1/2$, satisfying the trace formula \eqref{trvinsn} with $n=2$ for bounded rational functions $f$ was established in \cite[Theorem 2]{Koplienko84} under the assumption $V|H-iI|^{-\frac12}\in\S^2$. A more relaxed condition $(H+V-iI)^{-1}-(H-iI)^{-1}\in\S^n$ was traded off for a more restrictive set of functions $f$ and, when $n\ge 2$, for more complicated trace formulas where both the left and right hand sides of \eqref{trvinsn} are modified. The respective results for $n=1$ can be found in \cite[Theorem 3]{Krein62} and \cite[Theorem 2.2]{Yafaev05}; for $n=2$ in \cite[Theorem 3.5, Corollary 3.6]{Neidhardt88}; for $n\ge 2$ in \cite[Theorem 3.5]{PSS15} and \cite{S17}.

\medskip

\paragraph{\bf Methods.}
The major technical tools and novelty of our approach are briefly discussed below.

The technical scheme leading to the representation \eqref{trvinsn} under the assumption \eqref{vresinsn} is more subtle
than the one under the assumption \eqref{vinsn}. The derivatives and Taylor approximations of operator functions are known to be expressible in terms of multiple operator integrals (see Theorems \ref{dm} and \ref{rm}). The prime technique to handle these multiple operator integrals (see Theorem \ref{thm:PSS Higher Order}) only applies to compact perturbations satisfying \eqref{vinsn}. To bridge the gap between existing results for \eqref{vinsn} and our setting \eqref{vresinsn} we impose suitable weights on the perturbations and involve multi-stage approximation arguments for functions and perturbations.

In Theorem~\ref{thm:adding resolvents} we create Schatten class perturbations out of relative Schatten class perturbations \eqref{vresinsn} inside a multiple operator integral whose integrand is the $n$th order divided difference $f^{[n]}$ of a function $f\in C^n(\R)$ satisfying the properties $f^{(k)}(x)=o(|x|^{-k})$ as $|x|\rightarrow\infty$, $k=0,\ldots,n$, and $\widehat{f^{(n)}}\in L^1(\R)$.

Our Theorem \ref{thm:adding resolvents} significantly generalizes and extends earlier attempts in that direction made in \cite[Lemma 3.6]{S14}, \cite[Proposition 2.7]{S18}, \cite[Lemma 4.1]{CS18}.

The proof of Theorem \ref{thm:adding resolvents} involves the introduction of novel function classes (see Definition \ref{fracw}, \eqref{mBn}, and \eqref{mbn}), approximation arguments (see Lemma \ref{lem:density}), and analysis of multilinear operator integrals.

Based on the aforementioned results and analysis of distributions,
in Proposition \ref{prop:SSM with growth} we establish the trace formula
\begin{align}
\label{trexist}
\Tr\left(f(H+V)-\sum_{k=0}^{n-1}\frac{1}{k!}\frac{d^k}{dt^k} f(H+tV)\big|_{t=0}\right)
=\int_\R f^{(n)}(x)\,d\mu_n(x)
\end{align}
for every $f\in\mW_n$, where $\mu_n$ is a Borel measure determined uniquely up to an absolutely continuous term whose density is a polynomial of degree at most $n-1$ and such that for every $\epsilon>0$ the measure $(x-i)^{-n-\epsilon}\,d\mu_n(x)$ is finite and satisfies
\begin{align}
\|(\cdot-i)^{-n-\epsilon}\,d\mu_n\|\le c_n\,(1+\epsilon^{-1})(1+\norm{V})\nrm{V(H-iI)^{-1}}{n}^n.
\end{align}

In order to obtain absolute continuity of $\mu_n$ (and hence obtain a spectral shift \textit{function}) we apply the change of variables provided by Theorem~\ref{thm:adding resolvents} again, this time to multiple operator integrals of order $n-1$. This entails new terms for which the trace is defined only when perturbations satisfy additional summability requirements.
We establish an auxiliary result for finite rank perturbations in Proposition \ref{prop36} and then extend it
to relative Schatten class perturbations appearing in our main result with help of two new approximation results, one for operators obtained in Lemma \ref{lem:approximating V by V_k} and the other for Taylor remainders obtained in Lemma \ref{prop:|Tr(R)|}. In order to apply those approximation results, in Lemma \ref{lem:pth part of remainder} we derive a new representation for the remainder of the Taylor approximation of $f(H+V)$ in terms of handy components that are continuous in $V$ in a very strong sense.

In order to strengthen \eqref{trexist}, in Proposition \ref{prop:SSF locally} we establish another weaker version of \eqref{trvinsn} for $f\in C_c^{n+1}(\R)$, where on the left hand side we have a certain component of the Taylor remainder and on the right hand side in place of $f$ we have its product with some complex weight. By combining advantages of the results of Propositions \ref{prop:SSM with growth} and \ref{prop:SSF locally} we derive the trace formula \eqref{trvinsn}.
\medskip

\paragraph{\bf Examples.}

The relative Schatten class condition \eqref{vresinsn} arises in noncommutative geometry; see, for instance, \cite{WvS,SZ18}.
In that setting, $H$ is a generalized Dirac operator occuring in a (possibly non-unital) spectral triple and $V$ a generalized vector potential \cite[Section IV.1]{Connes}, which is also known as an inner fluctuation or Connes' differential one-form \cite{CC97,WvS}. For unital spectral triples, the condition \eqref{rinsn}, which is known as finite summability, is often assumed. For non-unital spectral triples, conditions similar to \eqref{inner perturbations} are discussed in Section \ref{sec4a}. Both in the unital and non-unital case, it is important to relax assumptions on the function $f$ appearing in the spectral action \cite{CC97} since that function might be prescribed by the model \cite{CCS}. Sometimes it is impossible or at least inconvenient to assume that $f$ is  given by a Laplace transform, as it was done in \cite{WvS}, and a general class of functions considered in this paper is more beneficial.

The condition \eqref{vresinsn} is also satisfied by many Dirac as well as random and deterministic Schr\"{o}dinger operators $H$ with $L^p$-potentials $V$. Appearance of such operators in problems of mathematical physics is discussed in, for instance, \cite{S21,Yafaev10} and references cited therein.
Sufficient conditions for \eqref{vresinsn} are discussed in Section \ref{sec4}.

\section{Notations}

\label{intro}
Let $\H$ be a separable Hilbert space, $\mB(\H)$ the $C^*$-algebra of all bounded linear operators on $\H$, and $\mB(\H)_{\text{sa}}$ the subset of all self-adjoint operators in $\mB(\H)$. For $p\in[1,\infty)$ we denote the respective Schatten-von Neumann ideal of compact operators on $\H$ by $\S^p$ and briefly call it the Schatten $p$-class. Basic properties of Schatten-von Neumann ideals can be found in, for instance, \cite{Simon05,ST19}. In some cases it will also be convenient to denote $\S^\infty:=\mB(\H)$.

Let $\N$ denote the positive natural numbers and let $n\in\N$.
When $H$ is a self-adjoint operator densely defined in $\H$, we briefly write
{\it $H$ is a self-adjoint operator in $\H$}.
Given a self-adjoint operator $H$ in $\H$ and $V\in\mB(\H)$, we denote
$$\tilde{V}:=V(H-iI)^{-1}.$$
Throughout the paper we will also use the notations $$u(\lambda):=\lambda-i$$ and $u^{-k}(x):=(u(x))^{-k}$. If $H_0,\ldots,H_m$ are self-adjoint operators in $\H$, and $V_1,\ldots,V_m$ are bounded operators, we denote
$$\tilde{V}_j:=V_ju^{-1}(H_j)=V_j(H_j-iI)^{-1}.$$

Given two Banach spaces $\mathcal{X}$ and $\mathcal{Y}$, let $\mB(\mathcal{X},\mathcal{Y})$ denote the Banach space of all bounded linear operators mapping $\mathcal{X}$ to $\mathcal{Y}$. For $T\in\mB(\mathcal{X},\mathcal{Y})$, we denote its norm by $\|T\|_{\mathcal{X}\to\mathcal{Y}}$.

We denote positive constants by letters $c,C$ with superscripts indicating dependence on their parameters. For instance, the symbol $c_\alpha$ denotes a constant depending only on the parameter $\alpha$.

\paragraph{\bf Function Spaces.}\label{sct:a new function space}
Let $C_0=C_0(\R)$ denote the space of continuous functions on $\R$ decaying to $0$ at infinity, $C_c=C_c(\R)$ the space of compactly supported continuous functions on $\R$, $C_c^n$ the class of $n$ times continuously differentiable functions in $C_c$, and $C^n_c[-a,a]$ the class of functions in $C_c^n$ whose support is contained in $[-a,a]$.
Let $C_b^n$ denote the subset of $C^n$ of such $f$ for which $f^{(n)}$ is bounded and let $C_0^n$ denote the subset of $C^n$ of such $f$ for which $f^{(n)}\in C_0(\R)$. We write $f(x)=\O(g(x))$ if there exists $M>0$ such that $|f(x)|\leq Mg(x)$ for all $x$ outside a compact set. We write $f(x)=o(g(x))$ if for all $\epsilon>0$, we have $|f(x)|\leq\epsilon g(x)$ for all $x$ outside a compact set depending on $\epsilon$.

Let $L^p$ denote the space of measurable $f$ for which $|f|^p$ is Lebesgue integrable on $\R$ equipped with the standard norm $\|f\|_p=\|f\|_{L^p}:=(\int_\R|f(x)|^p\,dx)^{1/p}$, $1\le p<\infty$, and let $L^\infty$ denote the space of essentially bounded functions on $\R$ equipped with the ${\rm ess\,sup}$ norm $\|\cdot\|_\infty$.
Let $L^1_{\text{loc}}$ denote the space of functions locally integrable on $\R$ equipped with the seminorms $f\mapsto \int_{-a}^a|f(x)|\,dx$, $a>0$. By $\ell^p(L^2(\R^d))$, where $p\ge 1$, we denote the space of functions consisting of those measurable functions $f:\R^d\to\C$ for which
\begin{align}
\label{lLdef}
\|f\|_{\ell^p(L^2(\R^d))}^p
:=\sum_{k\in\Z^d}\Big(\int\limits_{(0,1)^d+k}|f(x)|^2\,dx\Big)^{\frac{p}{2}}
<\infty.
\end{align}

Whenever we write $\hat{f}\in L^1$, it is implicitly assumed that $f\in C_0\subseteq\S'$, in order to define the Fourier transform. This can be done without loss of generality by the Riemann-Lebesgue lemma.

We recall that the divided difference\index{divided difference $f^{[n]}$} of the zeroth order $f^{[0]}$ is the function $f$ itself. Let $\la_0,\la_1,\dots,\la_n$ be points in $\R$ and let $f\in C^n(\R)$. The divided difference $f^{[n]}$ of order $n$ is defined recursively by
\begin{align*}
f^{[n]}(\la_0,\dots,\la_n)
=\lim\limits_{\la\rightarrow\la_n}\frac{f^{[n-1]}(\la_0,\dots,\la_{n-2},\la)
-f^{[n-1]}(\la_0,\dots,\la_{n-2},\la_{n-1})}{\la-\la_{n-1}}.
\end{align*}

\section{Auxiliary technical results}

In this section we set a technical foundation for the proof of our main result.

\subsection{New function classes}
\label{sec2}

In this subsection we introduce a new class of functions $\mW_n$, for which our main result holds, along with auxiliary classes $\mathfrak{B}_n$ and $\mathfrak{b}_n$ and derive their properties.

\begin{defn}
\label{fracw}
Let $\mathfrak{W}_n$ denote the set of functions $f\in C^n(\R)$ such that
\begin{enumerate}[(i)]
\item $\widehat{f^{(k)}u^k}\in L^1(\R),~k=0,\ldots,n$,
\item\label{fracwii} $f^{(k)}\in L^1\big(\R,(1+|x|)^{k-1}\,dx\big)$, $k=1,\ldots,n$.
\end{enumerate}
\end{defn}

The following sufficient condition for integrability of the Fourier transform of a function is a standard exercise and, thus, its proof is omitted.

\begin{lem}\label{lem:FT in L1}
If $f\in L^2(\R)\cap C^1(\R)$ and $f'\in L^2(\R)$, then $\hat{f}\in L^1(\R)$.
\end{lem}

\begin{prop}
\label{inclusions}
Let $n\in\N$. Then, the following assertions hold.
\begin{enumerate}[(i)]
\item\label{inclusionsi}
For every $\alpha>\frac12$,
\begin{align*}
\mW_n\supseteq\left\{f\in C^{n+1}:~ f^{(k)}(x)=\O\left(|x|^{-k-\alpha}\right)
\text{ as } |x|\to\infty,\;k=0,\ldots,n+1\right\}.
\end{align*}

\item\label{inclusionsii}
Furthermore,
\begin{align*}
\mW_n\subseteq\left\{f\in C^n:\;f^{(k)},\widehat{f^{(k)}}\in L^1(\R),\; k=1,\ldots,n\right\}.
\end{align*}
\end{enumerate}
\end{prop}

\begin{proof}
The inclusion in (i) is straightforward, as it follows from Lemma \ref{lem:FT in L1}.

(ii) The properties $f^{(k)}\in L^1(\R),$ $k=1,\ldots,n$ follow immediately from the definition of $\mW_n$. To prove $\widehat{f^{(k)}}\in L^1(\R)$, $k=1,\dots,n$, firstly we note that
\begin{align*}
\widehat{u^{-k}}\in L^1,\quad k=1,\ldots n
\end{align*}
by Lemma \ref{lem:FT in L1}. By the convolution theorem we find
\begin{align*}
\widehat{f^{(k)}}=\widehat{f^{(k)}u^k}*\widehat{u^{-k}},\quad k=1,\ldots,n,
\end{align*}
which is in $L^1$ because $L^1$ is closed under the convolution product. Therefore, the proof of (ii) is complete.
\end{proof}

\begin{remark}
\label{mncontains}
It follows from Proposition \ref{inclusions}\eqref{inclusionsi} that
$\mW_n$ contains all bounded rational functions except for linear combinations with constant functions, which are trivial in the context of our paper. In particular, $\mW_n$ contains the space ${\rm span}\,\{(z-\,\cdot)^{-k},\; \Im(z)>0,\; k\in\N,\; k\ge 2n\}$ considered in \cite{CS18}. In addition, $\mW_n$ contains all Schwartz functions and every $f\in C^{n+1}$ such that $f(x)=|x|^{-\alpha}$ outside a bounded neighborhood of zero for some $\alpha>\frac12$.
\end{remark}

We will need the auxiliary function classes
\begin{align}
\label{mBn}
\mathfrak{B}_n:=&\left\{f\in C^n:~f^{(k)}u^k\in C_0(\R),~k=0,\ldots,n,~\widehat{f^{(n)}}\in L^1(\R)\right\}
\end{align}
and
\begin{align}
\label{mbn}
\mathfrak{b}_n:=&\left\{f\in \mathfrak{B}_n:~\widehat{f^{(p)}u^p}\in L^1(\R),~p=0,\ldots,n\right\}.
\end{align}
It follows from Definition \ref{fracw} and Proposition \ref{inclusions}\eqref{inclusionsii} that
\begin{align*}
\mW_n\subset\mathfrak{B}_n.
\end{align*}
We also have the following result relating $\mathfrak{b}_n$ and $\mathfrak{B}_n$.

\begin{lem}\label{lem:density}
	The space $\mathfrak{b}_n$ is dense in $\mathfrak{B}_n$ with respect to the norm
		$$\norm{f}_{\mathfrak{B}_n}:=\sum_{p=0}^n\supnorm{f^{(p)}u^p}+\absnorm{\widehat{f^{(n)}}}.$$
\end{lem}
\begin{proof}
	Let $f\in\mathfrak{B}_n$. Fix a Schwartz function $\phi$ such that $\hat{\phi}\in C^\infty_c(\R)$ and $\phi(0)=1$. For every $k\in\N$, define
\begin{align*}
\phi_k(x):=\phi(x/k),\quad x\in\R.
\end{align*}
We note that $\big\{\,\widehat{\phi_k}\,\big\}_{k=1}^\infty$ is an approximate identity.
In particular, it satisfies the property
\begin{align}
\label{ai1}
\|\widehat{\phi_k}*g-g\|_1\rightarrow 0\quad\text{as }\;k\rightarrow\infty
\end{align}
for every $g\in L^1$.
Define
$$f_k:=\phi_kf.$$
Because every $\phi_k^{(m)}$ is of rapid decrease, it is obvious that $f_k^{(p)}u^p=\sum_{m=0}^p \vect{p}{m} \phi_k^{(m)}f^{(p-m)}u^p$ is integrable for every $p\in\{0,\ldots,n\}$.
	By Lemma \ref{lem:FT in L1} and the rapid decrease of every $\phi_k^{(m)}$, we obtain that $\widehat{f_k^{(p)}u^p}\in L^1$ for every $p\in\{0,\ldots,n-1\}$.
	In the same way, we obtain that $(f^{(p)}\phi_k^{(n-p)}u^n)\hat{~}\in L^1$ for every $p\in\{0,\ldots,n-1\}$.
	Moreover, we have $(f^{(n)}\phi_ku^n)\hat{~}=\widehat{f^{(n)}}*\widehat{\phi_ku^n}\in L^1$. Hence,
		$$\widehat{f_k^{(n)}u^n}=\sum_{p=0}^n\vect{n}{p}(f^{(p)}\phi_k^{(n-p)}u^n)\hat{~}\in L^1.$$
	We conclude that $f_k\in\mathfrak{b}_n$.

	In order to prove that $\|f^{(p)}u^p-f_k^{(p)}u^p\|_\infty\to0$ as $k\to\infty$, we write
	\begin{align}\label{eq:supnorm bounds phi_k}
		\supnorm{f^{(p)}u^p-f_k^{(p)}u^p}\leq \supnorm{(1-\phi_k)f^{(p)}u^p}+\sum_{m=1}^p\vect{p}{m}\supnorm{\phi_k^{(m)}u^mf^{(p-m)}u^{p-m}}.
	\end{align}
Since $f^{(p)}u^p\in C_0(\R)$, we obtain
\begin{align}
\label{term1}
\supnorm{(1-\phi_k)f^{(p)}u^p}\rightarrow 0\quad\text{as }\, k\rightarrow\infty.
\end{align}
By using $\phi^{(m)}_k(x)=\phi^{(m)}(x/k)/k^m$, we obtain
\begin{align}
\label{phik}
|\phi_k^{(m)}(x)u^m(x)|\leq\sqrt{2}^{\,m}\supnorm{\phi^{(m)}}k^{-m/2}\quad\text{for }\, x\in[-\sqrt{k},\sqrt{k}]
\end{align}
and
\begin{align}
\label{phiglobal}
\supnorm{\phi^{(m)}_ku^m}\leq \sqrt{2}^{\,m}\supnorm{\phi^{(m)}u^m}.
\end{align}
We now analyze the terms on the right hand side of \eqref{eq:supnorm bounds phi_k} as $k\rightarrow\infty$.
By \eqref{phik}, \eqref{phiglobal}, and the assumption $f^{(p-m)}u^{p-m}\in C_0$, we obtain $\|\phi_k^{(m)}u^mf^{(p-m)}u^{p-m}\|_\infty\to0$ as $k\rightarrow\infty$.
Combining the latter with \eqref{eq:supnorm bounds phi_k} and \eqref{term1} implies
$$\supnorm{f^{(p)}u^p-f_k^{(p)}u^p}\to0\quad\text{as }\, k\rightarrow\infty,\quad p=0,\dots,n.$$

We are left to prove that $\|\widehat{f^{(n)}}-\widehat{f_k^{(n)}}\|_1\to0$.
Applying $f_k^{(n)}=\sum_{m=0}^n \begin{psmallmatrix}n\\ m\end{psmallmatrix}\phi_k^{(m)}f^{(n-m)}$ along with standard properties of the Fourier transform and convolution yields
\begin{align}
\label{above}
\absnorm{\widehat{f^{(n)}}-\widehat{f_k^{(n)}}}\leq\absnorm{\widehat{f^{(n)}}-\widehat{\phi_k}*\widehat{f^{(n)}}}+\sum_{m=1}^n\vect{n}{m}\frac{\absnorm{\widehat{\phi^{(m)}}}}{k^m}\absnorm{\widehat{f^{(n-m)}}}.
\end{align}
The first term on the right hand side of \eqref{above} converges to $0$ as $k\rightarrow\infty$ by \eqref{ai1} applied to $g=\widehat{f^{(n)}}$. The other terms on the right hand side of \eqref{above} converge to $0$ as $k\rightarrow\infty$ because $1/k^m\to 0$.
\end{proof}

\subsection{Multilinear operator integration}
\label{sct:MOI}
In this subsection we recall known as well as establish new technical results on operator integration that are important in the proof of our main theorem. An interested reader can find a more detailed discussion of the known results in \cite{ST19}.

The following multilinear operator integral was introduced in \cite{PSS13} (see also \cite[Definition 4.3.3]{ST19}).

\begin{defn}
\label{nosep}
For $n\in\N$, let $\phi:\R^{n+1}\to\C$ be a bounded Borel function and fix $\alpha,\alpha_1,\ldots,\alpha_n\in [1,\infty]$ such that $\tfrac{1}{\alpha}=\tfrac{1}{\alpha_1}+\ldots+\tfrac{1}{\alpha_n}$.
Let $H_0,\dots,H_n$ be self-adjoint operators in $\H$.
Denote $E^j_{l,m}:=E_{H_j}\big(\big[\frac{l}{m},\frac{l+1}{m}\big)\big)$.
If for all $V_j\in\S^{\alpha_j}$, $j=1,\ldots,n$,
the iterated limit
$$T^{H_0,\ldots,H_n}_{\phi}(V_1,\ldots,V_n)
:=\lim_{m\to\infty}\lim_{N\to\infty}\sum_{|l_0|,\ldots,|l_n|<N}
\phi\left(\frac{l_0}{m},\ldots,\frac{l_n}{m}\right)E^0_{l_0,m}V_1E^1_{l_1,m}\cdots V_nE^n_{l_n,m}\,$$
exists in $\S^\alpha$, then the transformation $T^{H_0,\ldots,H_n}_{\phi}$, which belongs to $\mB(\S^{\alpha_1}\times\cdots\times\S^{\alpha_n},\S^\alpha)$ by the Banach-Steinhaus theorem, is called a multilinear operator integral.
\end{defn}

We write $T^{H_0,\ldots,H_n}_\phi\in\mB(\S^{\alpha_1}\times\cdots\times\S^{\alpha_n},\S^\alpha)$ to indicate that $T^{H_0,\ldots,H_n}_\phi$ exists in the sense of Definition \ref{nosep}. The transformation given by the latter definition satisfies the following powerful estimate.

\begin{thm}\label{thm:PSS Higher Order}
Let $\alpha,\alpha_1,\ldots,\alpha_n\in (1,\infty)$. If $f\in C^n$ is such that $f^{(n)}\in C_b$, then $T^{H_0,\ldots,H_n}_{f^{[n]}}\in\mB(\S^{\alpha_1}\times\cdots\times\S^{\alpha_n},\S^\alpha)$ and
\begin{align}
\big\|T^{H_0,\ldots,H_n}_{f^{[n]}}\big\|_{\S^{\alpha_1}\times\cdots\times\S^{\alpha_n}\to\S^\alpha}\leq c_{\alpha_1,\ldots,\alpha_n}\supnorm{f^{(n)}}.
\end{align}	

\end{thm}
\begin{proof}
The result for $H_0=\ldots=H_n$ is proved in \cite[Theorem 5.6]{PSS13}. Its extension to the case of distinct $H_0,\ldots,H_n$ is explained in the proof of \cite[Theorem 4.3.10]{ST19}.
\end{proof}

The domain of $T^{H_0,\ldots,H_n}_{\phi}$ extends to $\mB(\H)^{\times n}=\S^\infty\times\cdots\times\S^\infty$ for functions $\phi$ admitting a certain separation of variables. The proof of the following result can be found in \cite[Lemma 3.5]{PSS13}.

\begin{thm}\label{thm:coincidence}
Let $H_0,\ldots,H_n$ be self-adjoint operators in $\H$. Let $\phi:\R^{n+1}\to\C$ be a function admitting the representation
\begin{align}\label{phi representation}
	\phi(\lambda_0,\ldots,\lambda_n)=\int_\Omega\alpha_0(\lambda_0,s)\cdots\alpha_n(\lambda_n,s)\,d\nu(s),
\end{align}
where $(\Omega,\nu)$ is a finite measure space,
$$\alpha_j(\cdot,s):\R\to\C,\quad s\in\Omega,$$ are bounded continuous functions, and there is a sequence $\{\Omega_k\}_{k=1}^\infty$ of growing measurable subsets
of $\Omega$ such that $\Omega=\cup_{k=1}^\infty\Omega_k$ and the families
$$\{\alpha_j(\cdot,s)\}_{s\in\Omega_k},\quad j=0,\dots,n$$
are uniformly bounded and uniformly equicontinuous. Then,
$T^{H_0,\ldots,H_n}_\phi\in\mB( \S^{\alpha_1}\times\cdots\times\S^{\alpha_n},\S^\alpha)$ for all $\alpha,\alpha_j\in[1,\infty]$ with $\tfrac{1}{\alpha_1}+\ldots+\tfrac{1}{\alpha_n}=\tfrac{1}{\alpha}$, as well as
\begin{align*}
T^{H_0,\ldots,H_n}_\phi(V_1,\ldots,V_n)(y)=\int_\Omega\alpha_0(H_0,s)V_1\alpha_1(H_1,s)\cdots V_n\alpha_n(H_n,s)y\,d\nu(s),\quad y\in\H,
\end{align*}
and
\begin{align*}
\big\|T_{\phi}^{H_0,\dots,H_n}\big\|_{\S^{\alpha_1}\times\cdots\times\S^{\alpha_n}\to\S^\alpha}
\le \inf\int_\Omega\prod_{j=0}^n\|\alpha_j(\cdot,s)\|_\infty\,d|\nu|(s),
\end{align*}
where the infimum is taken over all possible representations \eqref{phi representation}.
\end{thm}
We will also need the following particular case of Theorem \ref{thm:coincidence}.
\begin{thm}\label{thm:FT-L1-bound of MOI}
If $f\in C^n$ and $\widehat{f^{(n)}} \in L^1$, then $\phi=f^{[n]}$ satisfies the assumptions of Theorem \ref{thm:coincidence} and, for all $\alpha,\alpha_j\in[1,\infty]$ with $\tfrac{1}{\alpha_1}+\ldots+\tfrac{1}{\alpha_n}=\tfrac{1}{\alpha}$,
\begin{align}
\label{fourierbound}
\big\|T_{f^{[n]}}^{H_0,\dots,H_n}\big\|_{\S^{\alpha_1}\times\cdots\times\S^{\alpha_n}\to\S^\alpha}
\le \frac{1}{n!}\big\|\widehat{f^{(n)}}\big\|_1.
\end{align}
\end{thm}
\begin{proof}
Let $\phi=f^{[n]}$, where $f\in C^n$ and $\widehat{f^{(n)}} \in L^1$. A straightforward induction argument (see, e.g., the proofs of \cite[Lemma 5.1 and Lemma 5.2]{PSS13}) gives
\begin{align}
\label{ddfourier}
f^{[n]}(\lambda_0,\ldots,\lambda_n)=\int_{\Delta_n}\int_\R e^{its_0\lambda_0}\cdots e^{its_n\lambda_n}\widehat{f^{(n)}}(t)\,dt\,d\sigma(s),
\end{align}
where $\Delta_n=\big\{s=(s_0,\dots,s_n)\in\R_{\ge 0}^{n+1}:\; \sum_{j=0}^n s_j=1\big\}$ is the $n$-simplex, $d\sigma$ is the Lebesgue measure on $\Delta_n$, and $dt$ is the Lebesgue measure on $\R$.
That is, $f^{[n]}$ admits a representation of the form \eqref{phi representation}, where
$(\Omega,\nu)=\big(\Delta_n\times\R,d\sigma\times\big(\widehat{f^{(n)}}(t)\,dt\big)\big)$. Since $\big\|d\sigma\times\big(\widehat{f^{(n)}}(t)\,dt\big)\big\|\le\tfrac{1}{n!}\|\widehat{f^{(n)}}\|_1$,
the estimate
\eqref{fourierbound} follows.
\end{proof}

All three of the above known theorems (Theorems \ref{thm:PSS Higher Order}, \ref{thm:coincidence}, \ref{thm:FT-L1-bound of MOI}) are needed to prove the following new crucial result, Theorem \ref{thm:adding resolvents}. That theorem  creates Schatten class perturbations $\tilde V_j=V_j(H_j-iI)^{-1}$ out of relative Schatten class perturbations $V_j$ inside a multiple operator integral by means of a certain change of variables.
It will be used throughout this paper, in particular to apply the bound from Theorem \ref{thm:PSS Higher Order} to the relative Schatten case, in which the perturbation $V$ is generally noncompact.

\begin{thm}\label{thm:adding resolvents}
Let $n\in\N$, let $H_0,\ldots,H_n$ be self-adjoint operators in $\H$, and let $V_1,\ldots,V_n\in\mB(\H)$. Then, the multiple operator integral given by Definition \ref{nosep} satisfies the following properties.
\begin{enumerate}[(i)]
\item\label{weight}
For every $f\in C^{n}$ satisfying $\widehat{f^{(n)}},\widehat{(fu)^{(n)}},\widehat{f^{(n-1)}}\in L^1$, we have
\begin{align}
\label{weight0}
 T^{H_0,\ldots,H_n}_{f^{[n]}}(V_1,\ldots,V_n)
=& T^{H_0,\ldots,H_n}_{(fu)^{[n]}}((H_0-iI)^{-1}V_1,V_2,\ldots,V_n)\\
\nonumber
&-(H_0-iI)^{-1}V_1T^{H_1,\ldots,H_n}_{f^{[n-1]}}(V_2,\ldots,V_n),\\
\label{weight1}
T^{H_0,\ldots,H_n}_{f^{[n]}}(V_1,\ldots,V_n)
=&T^{H_0,\ldots,H_n}_{(fu)^{[n]}}(V_1,\ldots,\tilde{V}_j,\ldots,V_n)\\
\nonumber
&-T^{H_0,\ldots,H_{j-1},H_{j+1},\ldots,H_n}_{f^{[n-1]}}(V_1,\ldots,\tilde{V}_jV_{j+1},\ldots,V_n)
\end{align}
for $j=1,\ldots,n-1$, and
\begin{align*}
T^{H_0,\ldots,H_n}_{f^{[n]}}(V_1,\ldots,V_n)
&=T^{H_0,\ldots,H_n}_{(fu)^{[n]}}(V_1,\ldots,V_{n-1},\tilde{V}_n)
-T^{H_0,\ldots,H_{n-1}}_{f^{[n-1]}}(V_1,\ldots,V_{n-1})\tilde{V}_n.
\end{align*}

\item\label{weights}
Denote $\tilde{V}_{j,l}:=\tilde V_{j+1}\cdots \tilde V_l$. Then, for all $f\in C_c^{n+1}$,
$$T^{H_0,\ldots,H_n}_{f^{[n]}}(V_1,\ldots,V_n)=\sum_{p=0}^n\sum_{0<j_1<\cdots<j_{p}\leq n}\!(-1)^{n-p}\, T^{H_0,H_{j_1},\ldots,H_{j_p}}_{(fu^p)^{[p]}}
(\tilde{V}_{0,j_1},\ldots,\tilde{V}_{j_{p-1},j_p})\,
\tilde{V}_{j_{p},n}.$$
If $V_k(H_k-iI)^{-1}\in\S^n$ for all $k=1,\ldots,n$, then the above formula holds for every $f\in\mathfrak{B}_n$ introduced in \eqref{mBn}, and hence, for every $f\in\mW_n$.

\item\label{trace class} If $V_k(H_k-iI)^{-1}\in\S^n$ for every $k=1,\dots,n$, then
$$T^{H_0,\ldots,H_n}_{f^{[n]}}(V_1,\ldots,V_n)\in\S^1$$
for every $f\in\mathfrak{b}_n$.
\end{enumerate}
\end{thm}
\begin{proof}
Since $u^{[1]}=1_{\R^2}$ and $u^{[p]}=0$ for all $p\geq2$, the Leibniz rule for divided differences gives
$$(fu)^{[n]}(\lambda_0,\ldots,\lambda_n)=f^{[n]}(\lambda_0,\ldots,\lambda_n)u(\lambda_n)+f^{[n-1]}(\lambda_0,\ldots,\lambda_{n-1}).$$
	If we swap $\lambda_n$ with $\lambda_j$ (for any $j\in\{0,\ldots,n\}$), and rearrange using symmetry of the divided difference, we obtain
\begin{align}\label{eq:adding one weight}
f^{[n]}(\lambda_0,\ldots,\lambda_n)
=&(fu)^{[n]}(\lambda_0,\ldots,\lambda_n)u^{-1}(\lambda_j)\\
\nonumber
&-f^{[n-1]}(\lambda_0,\ldots,\lambda_{j-1},\lambda_{j+1},\ldots,\lambda_n)u^{-1}(\lambda_j).
\end{align}
Applying \eqref{eq:adding one weight} repeatedly, we obtain
\begin{align}	
\label{eq:ddweights}	
&f^{[n]}(\lambda_0,\ldots,\lambda_n)\\
\nonumber
&=\sum_{p=0}^n\sum_{0<j_1<\cdots<j_{p}\leq n}(-1)^{n-p}
(fu^p)^{[p]}(\lambda_0,\lambda_{j_1},\ldots,\lambda_{j_p})
u^{-1}(\lambda_1)\cdots u^{-1}(\lambda_n).
\end{align}
Since $\widehat{f^{(n)}},\widehat{(fu)^{(n)}},\widehat{f^{(n-1)}}\in L^1$,
by Theorem \ref{thm:FT-L1-bound of MOI}, the functions $(fu)^{[n]}$ and $f^{[n-1]}$  admit the representation \eqref{phi representation}. Hence, the function on the right hand side of \eqref{eq:adding one weight} also admits the representation \eqref{phi representation}. Therefore, by Theorem \ref{thm:coincidence} applied
to $\phi=f^{[n]}$ and $\phi=\text{r.h.s of }\eqref{eq:adding one weight}$,
we obtain \eqref{weight}.
Similarly, applying Theorem \ref{thm:coincidence} and Theorem \ref{thm:FT-L1-bound of MOI} to \eqref{eq:ddweights} gives
\begin{align}\label{weights with hats}
	T^{H_0,\ldots,H_n}_{f_k^{[n]}}(V_1,\ldots,V_n)
=&\sum_{p=0}^n\sum_{0<j_1<\cdots<j_{p}\leq n}\!(-1)^{n-p}\, T^{H_0,H_{j_1},\ldots,H_{j_p}}_{(f_ku^p)^{[p]}}
(\tilde{V}_{0,j_1},\ldots,\tilde{V}_{j_{p-1},j_p})\,
\tilde{V}_{j_{p},n}\,
\end{align}
for all $f_k\in\mathfrak{b}_n$ introduced in \eqref{mbn}.

Let $f\in\mathfrak{B}_n$. By Lemma \ref{lem:density} we can choose $f_k\in\mathfrak{b}_n$ for all $k\in\N$ such that
\begin{align}\label{eq:density}
	\|\widehat{f_k^{(n)}}-\widehat{f^{(n)}}\|_1\to0\quad\text{and}\quad\|(f_ku^p)^{(p)}-(fu^p)^{(p)}\|_\infty\to0.
\end{align}
	The above $L^1$-norm-convergence implies that the left hand side of \eqref{weights with hats} converges (in operator norm) to $T^{H_0,\ldots,H_n}_{f^{[n]}}(V_1,\ldots,V_n)$ by \eqref{fourierbound}.
	Moreover, we find that $\tilde{V}_{j_{m-1},j_{m}}\in\S^{\alpha_m}$, where $\alpha_m:=n/(j_m-j_{m-1})\in(1,\infty)$ for $m=2,\ldots,p$, and $\tilde V_{0,j_1}\in\S^{\alpha_1}$, $\tilde V_{j_p,n}\in\S^{\alpha_{p+1}}$, where $\alpha_1=n/j_1\in[1,\infty)$, $\alpha_{p+1}=n/(n-j_p)\in(1,\infty]$.
By H\"{o}lder's inequality and Theorems \ref{thm:coincidence} and \ref{thm:FT-L1-bound of MOI}, we obtain that the right hand side of \eqref{weights with hats} is in $\S^1$, implying \eqref{trace class}.
On the strength of Theorem \ref{thm:PSS Higher Order} applied to $\S^{2\alpha_m}$, the supnorm-convergence in \eqref{eq:density} implies that the right hand side of \eqref{weights with hats} converges to the right hand side of \eqref{weights} in the operator norm (since convergence in Schatten norms implies uniform convergence). By uniqueness of limits in $\mB(\H)$, we conclude \eqref{weights}.
\end{proof}

\begin{remark}
(i) Although the condition $V(H-iI)^{-1}\in\S^n$ is equivalent to $V(H^2+I)^{-1/2}\in\S^n$, this paper makes use of the complex weight $u(t)=t-i$ rather than the real weight $\tilde u(t)=\sqrt{t^2+1}$ because there is no suitable analog of Theorem \ref{thm:adding resolvents} for the latter. For instance, an analog of \eqref{eq:adding one weight} for $\tilde{u}(t):=\sqrt{t^2+1}$ with $n=4$ and $j=1$ contains terms like
\begin{align}
f^{[2]}(\lambda_0,\lambda_2,\lambda_4)\,\tilde{u}^{[2]}(\lambda_1,\lambda_2,\lambda_3)\,u^{-1}(\lambda_1).
\end{align}
The latter is an obstacle to creating weights in the spirit of Theorem \ref{thm:adding resolvents}.
\end{remark}

\subsection{Taylor remainder via operator integrals}

The following two results are known. We refer the interested reader to \cite{ST19} for additional details.

\begin{thm}
\label{dm}
Let $n\in\N$ and let $f\in C^n(\R)$ be such that $\widehat{f^{(k)}}\in L^1(\R)$, $k=1,\dots,n$. Let $H$ be a self-adjoint operator in $\H$, let $V\in\mB(\H)_{\text{sa}}$.
Then, the Fr\'{e}chet derivative $\frac{1}{k!}\frac{d^k}{dt^k}f(H+tV)|_{t=0}$ exists in the operator norm and admits the multiple operator integral representation
\begin{align}
\label{dermoi}
\frac{1}{k!}\frac{d^k}{ds^k}f(H+sV)\big|_{s=t}=T_{f^{[k]}}^{H+tV,\dots,H+tV}(V,\dots,V).
\end{align}
The map $t\mapsto\frac{d^k}{ds^k}f(H+sV)|_{s=t}$
is continuous in the strong operator topology and, when $V\in\S^k$, in the $\S^1$-norm.
\end{thm}
\begin{proof}
The first assertion is given in \cite[Theorem 5.3.5]{ST19} and, in fact, holds for a larger set of functions.
The second assertion follows from \cite[Proposition 4.3.15]{ST19}. The proof relies on Theorems \ref{thm:coincidence} and \ref{thm:FT-L1-bound of MOI}.
\end{proof}

Given a function $f\in C^n(\R)$ satisfying $\widehat{f^{(k)}}\in L^1(\R)$, $k=1,\dots,n$, a self-adjoint operator $H$ in $\H$, and $V\in\mB(\H)_{\text{sa}}$, we denote the $n^{\text{th}}$ Taylor remainder by
\begin{align}\label{eq:def of the remainder}	R_{n,H,f}(V):=f(H+V)-f(H)-\sum_{k=1}^{n-1}\frac{1}{k!}\frac{d^k}{dt^k}f(H+tV)\big|_{t=0}.
\end{align}

The Taylor remainder admits the following representation in terms of a multiple operator integral.
\begin{thm}
\label{rm}
Let $n\in\N$ and let $f\in C^n(\R)$ be such that $\widehat{f^{(k)}}\in L^1(\R)$, $k=1,\dots,n$.
Let $H$ be a self-adjoint operator in $\H$, let $V\in\mB(\H)_{\text{sa}}$. Then,
\begin{align}
\label{remmoi}
R_{n,H,f}(V)=T^{H,H+V,H,\ldots,H}_{f^{[n]}}(V,\ldots,V),
\end{align}
where $T^{H,H+V,H,\ldots,H}_{f^{[n]}}$ is the multilinear operator integral given by
Definition \ref{nosep}.
\end{thm}
\begin{proof}
By \cite[Theorem 3.3.8]{ST19} for $k=0$ and \cite[Theorem 4.3.14]{ST19} for
$k\ge 1$,
\begin{align}
\label{pf}
&T^{H_0,H_1+V,H_2,\ldots,H_k}_{f^{[k]}}(V_1,\ldots,V_k)
-T^{H_0,\ldots,H_k}_{f^{[k]}}(V_1,\ldots,V_k)\\ \nonumber
&=T^{H_0,H_1+V,H_1,\ldots,H_k}_{f^{[k+1]}}( V_1,V,V_2,\ldots,V_k),
\end{align}
where $H_0,\ldots,H_k$ are self-adjoint operators in $\H$ and
$V,V_1,\ldots,V_k\in\mB(\H)_{\text{sa}}$. In particular,
\begin{align}\label{pf2}
T^{H,H+V,H,\ldots,H}_{f^{[k]}}(V,\ldots,V)-T^{H,\ldots,H}_{f^{[k]}}(V,\ldots,V)
=T^{H,H+V,H,\ldots,H}_{f^{[k+1]}}(V,\ldots,V).
\end{align}
Combining \eqref{pf2} with \eqref{dermoi} and proceeding by induction on $k$ yields \eqref{remmoi}.
\end{proof}

\section{Existence of the spectral shift function}
\label{sec3}

In this section we establish our main result.

\begin{thm}
\label{rsmain}
Let $n\in\N$, let $H$ be a self-adjoint operator in $\H$, and let $V\in\mB(\H)_{\text{sa}}$ be such that $V(H-iI)^{-1}\in\S^n$. Then, there exists $c_n>0$ and a real-valued function $\eta_n$
such that
\begin{align}\label{eta estimate}
\int_\R |\eta_n(x)|\,\frac{dx}{(1+|x|)^{n+\epsilon}}\leq c_n\,(1+\epsilon^{-1})(1+\norm{V})\nrm{V(H-iI)^{-1}}{n}^n\quad\text{for all } \epsilon >0
\end{align} and
\begin{align}
\label{tff}
\Tr(R_{n,H,f}(V))=\int_\R f^{(n)}(x)\eta_n(x)\,dx\,
\end{align}
for every $f\in\mW_n$. The locally integrable function $\eta_n$ is determined by \eqref{tff} uniquely up to a polynomial summand of degree at most $n-1$.
\end{thm}

We start by outlining major steps and ideas of the proof of Theorem \ref{rsmain}.

In Proposition \ref{prop:SSM with growth} we establish a weaker version of \eqref{tff} with  measure $d\mu_n$ on the right hand side of \eqref{tff} in place of the desired absolutely continuous measure $\eta_n(x)\,dx$. The measure $\mu_n$, which we call the spectral shift measure, satisfies the bound \eqref{eta estimate}. In Proposition \ref{prop:SSF locally} we establish another weaker version of \eqref{tff} for compactly supported $f$, where on the left hand side we have a certain component of the remainder and on the right hand side in place of $f$ we have its product with some complex weight. By combining advantages of the results of Propositions \ref{prop:SSM with growth} and \ref{prop:SSF locally} we derive the trace formula \eqref{tff}.

One of our main tools is multilinear operator integration developed for Schatten class perturbations. We have onset technical obstacles since our perturbations are not compact.
To bridge the gap between existing results and our setting we impose suitable weights on the perturbations and involve multistage approximation arguments. In particular, the proof of
Proposition \ref{prop:SSF locally} requires two novel techniques. The first one is a new expression for the remainder $R_{n,H,f}(V)$ in terms of handy components that are continuous in $V$ in a very strong sense. The second one is an approximation argument that allows replacing relative Schatten $V$ by finite rank $V_k$ and strengthens convergence arguments present in the literature.

\subsection{Existence of the spectral shift measure}

The following result is our first major step in the proof of the representation \eqref{tff}.

\begin{prop}\label{prop:SSM with growth}
Let $n\in\N$, let $H$ be a self-adjoint operator in $\H$, and let $V\in\mB(\H)_{\text{sa}}$ be such that $V(H-iI)^{-1}\in\S^n$.
Then, there exists a Borel measure $\mu_n$ such that
\begin{align}\label{mu tilde}
	\Tr(R_{n,H,f}(V))=\int_\R f^{(n)}\,d\mu_n\,
\end{align}
for every $f\in\mW_n$ and
\begin{align}
\label{munfla}
d\mu_n(x)=u^n(x)\,d\nu_n(x)+\xi_n(x)\,dx,
\end{align}
where $\nu_n$ is a finite measure satisfying
\begin{align}\label{eq:nu bound}
\|\nu_n\|\leq c_n(1+\norm{V})\nrm{V(H-iI)^{-1}}{n}^n,
\end{align}
and $\xi_n$ is a continuous function
satisfying
\begin{align}\label{eq:xi bound}
	|\xi_n(x)|\leq c_n(1+\norm{V})\nrm{V(H-iI)^{-1}}{n}^n(1+|x|)^{n-1},\quad x\in\R,
\end{align}
for some constant $c_n>0$.
If $\tilde\mu_n$ is another locally finite Borel measure such that \eqref{mu tilde} holds for all $f\in C^{n+1}_c$, then $d\tilde{\mu}_n(x)=d\mu_n(x)+p_{n-1}(x)\,dx$, where $p_{n-1}$ is a polynomial of degree at most $n-1$.
\end{prop}

To prove Proposition \ref{prop:SSM with growth} we need the estimate stated below.

\begin{lem}\label{lem:Hahn Banach Borel measures}
Let $k\in\N$, let $H_0,\ldots,H_k$ be self-adjoint operators in $\H$, let  $\alpha_1\ldots,\alpha_k\in(1,\infty)$ be such that $\tfrac{1}{\alpha_1}+\ldots+\tfrac{1}{\alpha_k}=1$, and
let $B_j\in\S^{\alpha_k}$, $j=1,\dots,k$.
Then, there exists $c_\alpha:=c_{\alpha_1,\dots,\alpha_k}>0$ such that  for multiple operator integrals given by Definition \ref{nosep},
\begin{align*}
|\Tr(T^{H_k,H_1,\ldots,H_k}_{f^{[k]}}(B_1,\ldots,B_k))|&\leq c_\alpha\supnorm{f^{(k)}}\nrm{B_1}{\alpha_1}\cdots\nrm{B_k}{\alpha_k}\qquad
(\widehat{f^{(k)}}\in L^1)
\end{align*}
and
\begin{align*}
|\Tr(B_1T^{H_0,\ldots,H_{k-1}}_{f^{[k-1]}}(B_2,\ldots,B_k))|&\leq c_\alpha\supnorm{f^{(k-1)}}\nrm{B_1}{\alpha_1}\cdots\nrm{B_k}{\alpha_k}\qquad (f\in C_b^{k-1}).
\end{align*}
Consequently, there exist unique (complex) Borel measures $\mu_1,\mu_2$ with total variation bounded by $c_\alpha\nrm{B_1}{\alpha_1}\cdots\nrm{B_k}{\alpha_k}$ such that
\begin{align*}
\Tr(T^{H_k,H_1,\ldots,H_k}_{f^{[k]}}(B_1,\ldots,B_k))&=\int_\R f^{(k)} d\mu_1\qquad
(\widehat{f^{(k)}}\in L^1)
\end{align*}
and
\begin{align*}
\Tr(B_1 T^{H_0,\ldots,H_{k-1}}_{f^{[k-1]}}(B_2,\ldots,B_k))&=\int_\R f^{(k-1)}d\mu_2\,\qquad (f\in C_0^{k-1}).
\end{align*}
\end{lem}

\begin{proof}
The first assertion of the lemma follows from \cite[Theorem 5.3 and Remark 5.4]{PSS13}, H\"{o}lder's inequality, and \cite[Theorem 4.3.10]{ST19}. The second assertion of the lemma is subsequently obtained by the Riesz--Markov representation theorem for a bounded linear functional on the space $C_0(\R)$.
\end{proof}

\begin{proof}[Proof of Proposition \ref{prop:SSM with growth}]
Let $n\geq2$.
Using \eqref{remmoi} and Theorem \ref{thm:adding resolvents}\eqref{weights}, we obtain
\begin{align}
\label{eq:new expansion of remainder}
\nonumber
R_{n,H,f}(V)=&T^{H,H+V,H,\ldots,H}_{f^{[n]}}(V,\ldots,V)\\
	=&\sum_{p=0}^n\sum_{\substack{j_1,\ldots,j_{p}\geq1,j_{p+1}\geq0\\j_1+\ldots+j_{p+1}=n}}(-1)^{n-p}\,
T^{H,H_{j_1},H,\ldots,H}_{(fu^p)^{[p]}}(\tilde{V}^{j_1},\ldots,\tilde{V}^{j_{p}})\tilde{V}^{j_{p+1}},
\end{align}
where $H_{1}=H+V$ and $H_{j_1}=H$ for $j_1\neq1$, and in which the first factor of $\tilde V$ in the first input of the multilinear
operator integral should be interpreted as $V(H+V-iI)^{-1}$.
By the second resolvent identity,
$$\|V(H+V-iI)^{-1}\|_n\leq (1+\norm{V})\|V(H-iI)^{-1}\|_n.$$
By the definition of $\mW_n$ (see Definition \ref{fracw}), 
we obtain $\widehat{(fu^p)^{(p)}}\in L^1(\R)$
for every $f\in\mW_n$, $p=0,\dots,n$. Hence, by Lemma \ref{lem:Hahn Banach Borel measures}
applied to each term of \eqref{eq:new expansion of remainder},
there exist unique Borel measures $\breve\mu_0,\ldots,\breve\mu_n$ such that
\begin{align}
\label{mupbound}
\|\breve\mu_p\|\le C_n\,(1+\norm{V})\,\|V(H-iI)^{-1}\|_n^n
\end{align}
and
\begin{align}\label{eq:Tr(R) in measures 1}
\Tr( R_{n,H,f}(V))=&\sum_{p=0}^n\int (fu^p)^{(p)}\,d\breve\mu_p
\end{align}
for every $f\in\mW_n$, $n\ge 2$.

Let $n=1$. Denote $H_t=H+tV$.
By Theorem \ref{dm}, continuity of the transformation $t\mapsto T^{H_t,H_t}_{f^{[1]}}(V)$
(see \cite[Proposition 3.3.9]{ST19}), and the fundamental theorem of calculus,
\begin{align*}
R_{1,H,f}(V)
=f(H+V)-f(H)=\int_0^1 T^{H_t,H_t}_{f^{[1]}}(V)\,dt
\end{align*}
for $f\in\mW_1$.
By \eqref{weight1} of Theorem \ref{thm:adding resolvents}\eqref{weight} applied to $T^{H_t,H_t}_{f^{[1]}}(V)$  we obtain
\begin{align}\label{eq:no trace yet}
R_{1,H,f}(V)
=\int_0^1 (T^{H_t,H_t}_{(fu)^{[1]}}(V(H_t-iI)^{-1})-f(H_t)V(H_t-iI)^{-1})\,dt.
\end{align}
Noticing that
\begin{align*}
\sup_{t\in [0,1]}\|V(H_t-iI)^{-1}\|_1\le(1+\|V\|)\|V(H-iI)^{-1}\|_1,
\end{align*}
using the property of the double operator integral $\Tr(T^{H,H}_{g^{[1]}}(V))=\Tr(g'(H)V)$,
and applying H\"{o}lder's inequality and the Riesz--Markov representation theorem
completes the proof of \eqref{eq:Tr(R) in measures 1} for $n=1$.

Let $n\in\N$. Applying a higher order differentiation product rule on the right hand side of \eqref{eq:Tr(R) in measures 1} gives
\begin{align}
\Tr( R_{n,H,f}(V))&=\sum_{p=0}^n\sum_{k=0}^p\vect{p}{k}\frac{p!}{k!}\int f^{(k)}u^k\,d\breve\mu_p\nonumber\\
	&=\sum_{k=0}^{n-1}\int f^{(k)}u^k\,d\grave{\mu}_k+\int f^{(n)}u^n\,d\nu_n,\label{Tr(R) in integrals}
\end{align}
for some Borel measures $\grave{\mu}_0,\ldots\grave\mu_{n-1},\nu_n$ satisfying
\begin{align}
\label{normnun} 
\|\grave\mu_0\|,\dots,\|\grave\mu_{n-1}\|,\|\nu_n\|\le \tilde C_n\,(1+\norm{V})\,\|V(H-iI)^{-1}\|_n^n.
\end{align}
Integrating by parts in \eqref{Tr(R) in integrals} and applying
\begin{align}
\label{partbounds}
\lim\limits_{x\rightarrow\pm\infty} f^{(k)}(x)u^k(x)=0,\quad k=0,\dots,n-1,
\end{align}
yields
\begin{align*}
\Tr(R_{n,H,f}(V))=-\sum_{k=0}^{n-1}\int_{-\infty}^\infty (f^{(k+1)}u^k+kf^{(k)}u^{k-1})(x)\,\grave{\mu}_k((-\infty,x))\,dx+\int f^{(n)}u^n\,d\nu_n.
\end{align*}
Since $$f^{(k)}u^{k-1}\in L^1(\R),\quad k=1,\dots,n,$$ we rearrange the terms above to obtain
\begin{align}
\label{bp1}
\Tr(R_{n,H,f}(V))=\sum_{k=1}^{n}\int f^{(k)}(x)u^{k-1}(x)\,\tilde\xi_k(x)\,dx+\int f^{(n)}u^n\,d\nu_n,
\end{align}
where $\tilde\xi_k$ are continuous functions defined by
\begin{align*}
&\tilde{\xi}_k(x)=-\grave\mu_{k-1}((-\infty,x))-k\,\grave\mu_k((-\infty,x)),\quad k=1,\dots,n-1,\\
&\tilde\xi_n(x)=-\grave\mu_{n-1}((-\infty,x)),
\end{align*}
so that
\begin{align}\label{eq:xi_k bound}
	\big\|\tilde\xi_k\big\|_\infty\leq c_{n,k}(1+\norm{V})\nrm{V(H-iI)^{-1}}{n}^n,\quad k=1,\dots,n.
\end{align}
By a repeated partial integration in \eqref{bp1} and application of \eqref{partbounds}, we obtain
\begin{align*}
\Tr(R_{n,H,f}(V))=\int_\R f^{(n)}\,d\mu_n\quad (f\in\mW_n)
\end{align*}
with
\begin{align}
\label{mugravexi}
d\mu_n(x)=u^n(x)\,d\nu_n(x)+\xi_n(x)\,dx,
\end{align}
where
\begin{align}\label{eq:def xi}	
\xi_n(s_0):=\sum_{k=1}^n(-1)^{n-k}\int_0^{s_0}ds_1\cdots\int_0^{s_{n-k-1}}\,
u^{k-1}(s_{n-k})\,\tilde\xi_k(s_{n-k})\,ds_{n-k}.
\end{align}
The function $\xi_n$ given by \eqref{eq:def xi} is continuous. To confirm \eqref{eq:xi bound} we note that,
for all $m\in\N$,
\begin{align}
\label{add u^-1}
\sup_{x\in\R}\left|u^{-m}(x)\int_0^xg(t)\,dt\right|\leq \sup_{x\in\R}\left(\left|\frac{x}{u(x)}\,\right|
\sup_{|t|\leq|x|}|u^{1-m}(x)g(t)|\right)\le\|u^{1-m}g\|_\infty.
\end{align}
By applying \eqref{add u^-1} $(n-k)$-times in \eqref{eq:def xi} and using the bound \eqref{eq:xi_k bound}, we obtain
\begin{align}\label{xiprelbound}
|\xi_n(x)|\leq c_n(1+\norm{V})\nrm{V(H-iI)^{-1}}{n}^n(1+|x|)^{n-1},\quad x\in\R.
\end{align}
We have thereby proven the first part of the proposition.

To prove the second part of the proposition, we let $\tilde{\mu}_n$ be a locally finite measure such that \eqref{mu tilde} holds for all $f\in C_c^{n+1}$
and denote
$$\rho_n:=\mu_n-\tilde\mu_n.$$
Then,
\begin{align}\label{eq:measure nth derivative}
	\int f^{(n)}\,d\rho_n=0\qquad(f\in C_c^{n+1}).
\end{align}
We are left to confirm that
\begin{align}
\label{goal}
d\rho_n(x)=p_{n-1}(x)\,dx,
\end{align}
where $p_{n-1}$ is a polynomial of degree at most $n-1$. Consider the distribution $T$ defined by 	
$$T(g):=\int g\, d\rho_n$$ for all $g\in C^\infty_c(\R)$.
By \eqref{eq:measure nth derivative} and the definition of the derivative of a distribution, $T^{(n)}=0$.
Since the primitive of a distribution is unique up to an additive constant (see, e.g., \cite[Theorem 3.10]{gwaiz}),
by an inductive argument (see, e.g., \cite[Example 2.21]{gwaiz}) we obtain \eqref{goal}.
\end{proof}

\subsection{Alternative trace formula}

The following result is our second major step in the proof of the representation \eqref{tff}.
It provides an alternative to \eqref{tff} with weighted $f$ on the right hand side.
It also provides an alternative to \eqref{mu tilde} with weighted $f$ on the right hand side, thereby effectively replacing the measure $\mu_n$ with functions $\breve\eta_0,\ldots,\breve\eta_{n-1}\in L^1_{\text{loc}}$.
\begin{prop}\label{prop:SSF locally}
Let $n\in\N$, $n\ge 3$, let $H$ be a self-adjoint operator in $\H$, and let $V\in\mB(\H)_{\text{sa}}$ satisfy $V(H-iI)^{-1}\in\S^n$. Then, for every $p=0,\dots,n-1$,
there exists $\breve\eta_p\in L^1_{\text{loc}}$ such that
\begin{align}
\label{trfp}
\Tr(R_{n,H,f}(V))=\sum_{p=0}^{n-1}(-1)^{n-1-p}\int_\R (fu^p)^{(p+1)}(x)\breve\eta_p(x)\,dx
\end{align}
for all $f\in C_c^{n+1}$.
\end{prop}

In order to prove \eqref{trfp} firstly we decompose $R_{n,H,f}(V)$ into more convenient components for which we can derive trace formulas by utilizing the method of the previous subsection, partial integration, and approximation arguments.


\begin{lem}\label{lem:pth part of remainder}
Let $H$ be a self-adjoint operator in $\H$, let $V\in\mB(\H)_{\text{sa}}$, let $n\in\N$, and let $f\in C_c^{n+1}$. Then, $$R_{n,H,f}(V)=\sum_{p=0}^{n-1} (-1)^{n-1-p}\tilde R^p_{n,H,f}(V),$$
where
\begin{align}
\label{R0}
\nonumber
&\tilde R^0_{1,H,f}(V):=f(H+V)-f(H),\\
&\tilde R^0_{n,H,f}(V):=f(H)V((H+V-iI)^{-1}-(H-iI)^{-1})\tilde V^{n-2}
\end{align}
for $n\geq 2$ and
\begin{align}
\label{rpdef}
\tilde R^p_{n,H,f}(V):=\sum_{\substack{j_1,\ldots,j_{p}\geq1,j_{p+1}\geq0\\j_1+\ldots+j_{p+1}=n-1}}\!\, \Big(& T^{H,H_{j_1},H,\ldots,H}_{(fu^p)^{[p]}}(V(H+V-iI)^{-1}\tilde{V}^{j_1-1}, \ldots,\tilde{V}^{j_p})\,\tilde{V}^{j_{p+1}}\nonumber\\
&-T^{H,\ldots,H}_{(fu^p)^{[p]}}(\tilde{V}^{j_1},\ldots,\tilde{V}^{j_p})\tilde{V}^{j_{p+1}}\Big)
\end{align}
for $ p=1,\dots, n-1$, with $H_1=H+V$ and $H_{j_1}=H$ for $j_1\neq1$.
\end{lem}

\begin{proof}\label{prop36}
Using \eqref{dermoi} and \eqref{remmoi}, we get
\begin{align}
R_{n,H,f}(V)=&R_{n-1,H,f}(V)-\frac{1}{(n-1)!}\frac{d^{n-1}}{dt^{n-1}}f(H+tV)|_{t=0}\nonumber\\
=&T^{H,H+V,H,\ldots,H}_{f^{[n-1]}}(V,\ldots,V)
-T^{H,\ldots,H}_{f^{[n-1]}}(V,\ldots,V)\label{birem}.
\end{align}
An application of Theorem \ref{thm:adding resolvents}\eqref{weights} to each of the terms in \eqref{birem} completes the proof.
\end{proof}

Firstly we show that \eqref{trfp} holds when $V$ is a finite-rank operator. This is done by establishing an analog of \eqref{trfp} for $\tilde R^p_{n,H,f}(V)$ and then extending \eqref{trfp} to $R_{n,H,f}(V)$ with help of Lemma \ref{lem:pth part of remainder}.

\begin{prop}\label{prop:SSF locally V_k}
Let $n\in\N$, $n\ge 3$, let $H$ be a self-adjoint operator in $\H$, and let $V\in\mB(\H)_{\text{sa}}$
be of finite rank. Then, for $p=0,\dots, n-1$, there exists $\breve\eta_{p}\in L^1_{\text{loc}}$ such that
\begin{align*}
\Tr(\tilde R^p_{n,H,f}(V))=\int_\R(fu^p)^{(p+1)}(x)\breve\eta_{p}(x)\,dx\,
\end{align*}
for all $f\in C^{n+1}_c$,  where $\tilde R^p_{n,H,f}$ is given by \eqref{rpdef}.
\end{prop}
\begin{proof}

By the definition of $\tilde R^p_{n,H,f}(V)$ in Lemma~\ref{lem:pth part of remainder},
\begin{align}
\label{sumabstrt}
&|\Tr(\tilde R^p_{n,H,f}(V))|\\ \nonumber
&\leq\sum_{\substack{j_1,\ldots,j_{p}\geq1,j_{p+1}\geq0\\j_1+\ldots+j_{p+1}=n-1}}
\Big(\big|\Tr\big(T^{H,H_{j_1},H,\ldots,H}_{(fu^p)^{[p]}}(V(H+V-iI)^{-1}\tilde{V}^{j_1-1},\ldots,
\tilde{V}^{j_p})\tilde{V}^{j_{p+1}}\big)\big|\\ \nonumber
&\quad+\big|\Tr\big(T^{H,\ldots,H}_{(fu^p)^{[p]}}
(\tilde{V}^{j_1},\ldots,\tilde{V}^{j_p})\tilde{V}^{j_{p+1}}\big)\big|\Big).
\end{align}
By Lemma \ref{lem:Hahn Banach Borel measures} applied to each summand on the right hand side of \eqref{sumabstrt},
\begin{align}	
|\Tr(\tilde R^p_{n,H,f}(V))|
\leq&\sum_{\substack{j_1,\ldots,j_{p}\geq1,j_{p+1}\geq0\\
j_1+\ldots+j_{p+1}=n-1}}2c_{n,j}\supnorm{(fu^p)^{(p)}}
(1+\|V\|)\nrm{V(H-iI)^{-1}}{n-1}^{n-1}\nonumber\\
=:&\,c_n\supnorm{(fu^p)^{(p)}}(1+\|V\|)\nrm{V(H-iI)^{-1}}{n-1}^{n-1}.
\label{eq:bound R^p}
\end{align}
Hence, by the Riesz-Markov representation theorem, there exist unique Borel measures $\breve\mu_{p}$ such that
$$\|\breve\mu_{p}\|\le c_n(1+\|V\|)\nrm{V(H-iI)^{-1}}{n-1}^{n-1}$$
and $$\Tr(\tilde R^p_{n,H,f}(V))=\int(fu^p)^{(p)}\,d\breve\mu_{p}$$
for all $f\in C^{n+1}_c\subseteq \mW_n$. Hence, $\eta_{p}(x):=-\breve\mu_{p}((-\infty,x))$ is a bounded function in $L^1_{\text{loc}}(\R)$ and the proposition follows by the partial integration formula for distribution functions.
\end{proof}

Proposition \ref{prop:SSF locally V_k} will be extended from finite rank to relative Schatten class perturbations by an approximation argument. To carry out the latter we build some technical machinery below.

We need the next standard result in operator theory.

\begin{lem}\label{lem:Schatten product is continuous}
Let $\alpha,\alpha_j\in[1,\infty]$ satisfy $\tfrac{1}{\alpha_1}+\ldots+\tfrac{1}{\alpha_n}=\tfrac{1}{\alpha}$.
Denote $\L^\alpha:=(\S^\alpha,\nrm{\cdot}{\alpha})$ for $\alpha\in[1,\infty)$ and $\L^\infty:=(\mB(\H)_1,\text{so*}\,)$, where $\mB(\H)_1$ denotes the closed unit ball in $\mB(\H)$.
Then, the function $$(A_1,\ldots,A_n)\mapsto A_1\cdots A_n$$
is a continuous map from $\L^{\alpha_1}\times\cdots\times\L^{\alpha_n}$ to $\L^\alpha$.
\end{lem}

The following approximation of weighted perturbations is an important step in the approximation of the trace formula given by Proposition \ref{prop:SSF locally V_k}.

\begin{lem}\label{lem:approximating V by V_k}
	Let $\H$ be a Hilbert space, $H$ a self-adjoint operator in $\H$, and let $V\in \mB(\H)_{\text{sa}}$ be such that $V(H-iI)^{-1}\in\S^n$. Then, there exists a sequence $(V_k)_k\subset\mB(\H)_{\text{sa}}$ of finite-rank operators such that $(V_k)_k$ converges strongly to $V$ and such that
\begin{align}\label{eq:V_k Schatten conv}
\nrm{V_k(H-iI)^{-1}-V(H-iI)^{-1}}{n}\to0\, \text{ as } k\rightarrow\infty\,	
\end{align}
and, moreover,
\begin{align}
\label{c=1}
\norm{V_k}\leq\norm{V}\quad\text{and}
\quad\nrm{V_k(H-iI)^{-1}}{n}\leq\nrm{V(H-iI)^{-1}}{n}.
\end{align}
\end{lem}
\begin{proof}
We start with a sequence of spectral projections, denoted $$P_k:=E_H((-k,k)),$$ which by the functional calculus converges strongly to $I$.
Applying subsequently the property of orthogonal projections and standard functional calculus we obtain
\begin{align}
\label{pk}
(P_kVP_k)((H-iI)^{-1}P_k+(I-P_k))
\nonumber
&=(P_kVP_k)((H-iI)^{-1}P_k)\\
&=P_kV(H-iI)^{-1}P_k\in\S^n
\end{align}
for each $k\in\N$.
By the functional calculus, $(H-iI)^{-1}P_k+(I-P_k)$ is invertible.
Therefore, from \eqref{pk} we derive $$P_kVP_k=P_kV(H-iI)^{-1}P_k\left((H-iI)^{-1}P_k+(I-P_k)\right)^{-1}\in\S^n.$$
For a fixed $k$, by the spectral theorem of compact self-adjoint operators, there exists a sequence $(E_l)_{l=1}^\infty$ of finite-rank projections, each $E_l$ commuting with $P_kVP_k$, such that $E_lP_kVP_k$ converges to $P_kVP_k$ in $\S^n$ as $l\to\infty$. For all $k\in\N$, there exists $l_k\in\N$ such that $$\nrm{E_{l_k}P_kVP_k-P_kVP_k}{n}<1/k.$$
Define $$V_k:=E_{l_k}P_kVP_k.$$ Then $\|V_k\|\leq\|V\|$ holds, $V_k$ is self-adjoint, $V_k\to V$ strongly, and
\begin{align*}
\nrm{V_k(H-iI)^{-1}-V(H-iI)^{-1}}{n}\leq& \nrm{E_{l_k}P_kVP_k-P_kVP_k}{n}\norm{(H-iI)^{-1}}\\
&+\nrm{P_kV(H-iI)^{-1}P_k-V(H-iI)^{-1}}{n}\to0.
\end{align*}
By Lemma \ref{lem:Schatten product is continuous}, the latter expression converges to $0$ as $k\rightarrow\infty$.
The second inequality in \eqref{c=1} follows from the estimate $$\nrm{E_{l_k}P_kVP_k(H-iI)^{-1}}{n}
\leq\norm{E_{l_k}}\norm{P_k}\nrm{V(H-iI)^{-1}}{n}\norm{P_k}.$$
\end{proof}


Our approximation on the left hand side of the trace formula in Proposition \ref{prop:SSF locally V_k} is based on the next estimate.
	
\begin{lem}\label{prop:|Tr(R)|}
Let $H$ be a self-adjoint operator in $\H$, let $n\in\N$, $n\neq2$, and let $V\in \mB(\H)_{\text{sa}}$ be such that $V(H-iI)^{-1}\in\S^n$.
Let $(V_k)_k\subset\mB(\H)_{\text{sa}}$ be a sequence satisfying the assertions of Lemma \ref{lem:approximating V by V_k}.
Let $W\in \{V,V_m\}$, where $m\in\N$.
Then, given $a>0$, there exists $c_{n,H,V,a}>0$ such that
$$|\Tr (\tilde R^p_{n,H,f}(V_k)-\tilde R^p_{n,H,f}(W))|\leq c_{n,H,V,a}\supnorm{(fu^p)^{(p+1)}}\|\tilde{V}_k-\tilde{W}\|_n$$	
for all $p=0,\dots,n-1$, $k\in\N$, and $f\in C^{n+1}_c[-a,a]$, where $\tilde R^p_{n,H,f}$ is given by \eqref{rpdef}.
In addition,
\begin{align*}
\Tr(R_{2,H,f}(V_k)-R_{2,H,f}(W))=\sum_{p=0}^2\int_0^1\Tr(R_{t,t,H,W,V_k,f}^p+R_{t,H,W,V_k,f}^p)\,dt
\end{align*}
for some operators $R_{t,t,H,W,V_k,f}^p$ and $R_{t,H,W,V_k,f}^p$ satisfying
$$|\Tr(R_{t,t,H,W,V_k,f}^p+R_{t,H,W,V_k,f}^p)|\leq c_{H,V}\|(fu^p)^{(p)}\|_\infty\|\tilde{V}_k-\tilde{W}\|_2$$
for all $f\in C^3_c[-a,a]$.
\end{lem}

\begin{proof}
Let $n\geq3$. By \eqref{rpdef} in Lemma \ref{lem:pth part of remainder},
\begin{align}
\label{rvw0}
&\tilde R^p_{n,H,f}(V_k)-\tilde R^p_{n,H,f}(W)\\
\nonumber
&=\sum_{\substack{j_1,\ldots,j_{p}\geq1,j_{p+1}\geq0\\j_1+\ldots+j_{p+1}=n-1}}\!\,
\Big(T^{H,H+V_{k,j_1},H,\ldots,H}_{(fu^p)^{[p]}}(V_{k}(H+V_{k}-iI)^{-1}\tilde{V}_k^{j_1-1}, \ldots,\tilde{V}_k^{j_p})\,\tilde{V}_k^{j_{p+1}}\\ \nonumber
&\quad-T^{H,H+W_{j_1},H,\ldots,H}_{(fu^p)^{[p]}}(W(H+W-iI)^{-1}\tilde{W}^{j_1-1}, \ldots,\tilde{W}^{j_p})\,\tilde{W}^{j_{p+1}}\\
\nonumber
&\quad-T^{H,\ldots,H}_{(fu^p)^{[p]}}(\tilde{V}_k^{j_1},\ldots,\tilde{V}_k^{j_p})\tilde{V}_k^{j_{p+1}}
+T^{H,\ldots,H}_{(fu^p)^{[p]}}(\tilde{W}^{j_1}, \ldots,\tilde{W}^{j_p})\,\tilde{W}^{j_{p+1}}\Big),
\end{align}
where $V_{k,1}=V_k$, $W_1=W$ and $V_{k,j}=W_j=0$ for $j\neq1$.
Below we also use the notations $\breve{V}_k^{j}=V_k(H+V_k-iI)^{-1}\tilde{V}_k^{j-1}$ and $\breve{W}^{j}=W(H+W-iI)^{-1}\tilde{W}^{j-1}$.

Firstly we handle the summands in \eqref{rvw0} with $j_1=1$.
By \eqref{pf},
\begin{align}
\label{rvw1}
&T^{H,H+V_k,H,\ldots,H}_{(fu^p)^{[p]}}(\breve{V}_k,\tilde{V}_k^{j_2}, \ldots,\tilde{V}_k^{j_p})\,\tilde{V}_k^{j_{p+1}}
-T^{H,H+W,H,\ldots,H}_{(fu^p)^{[p]}}(\breve{V}_k,\tilde{V}_k^{j_2},\ldots,\tilde{V}_k^{j_p})\tilde{V}_k^{j_{p+1}}\\
\nonumber
&=T^{H,H+V_k,H+W,H,\ldots,H}_{(fu^p)^{[p+1]}}
(\breve{V}_k,V_k-W,\tilde{V}_k^{j_2},\ldots,\tilde{V}_k^{j_p})\tilde{V}_k^{j_{p+1}}.
\end{align}
By telescoping we obtain
\begin{align}
\label{rvw2}
&T^{H,H+W,H,\ldots,H}_{(fu^p)^{[p]}}(\breve{V}_k,\tilde{V}_k^{j_2},\ldots,\tilde{V}_k^{j_p})\tilde{V}_k^{j_{p+1}}
-T^{H,H+W,H,\ldots,H}_{(fu^p)^{[p]}}(\breve{W},\tilde{W}^{j_2},\ldots,\tilde{W}^{j_p})\tilde{W}^{j_{p+1}}\\  \nonumber
&\quad-T^{H,\ldots,H}_{(fu^p)^{[p]}}(\tilde{V}_k,\tilde{V}_k^{j_2},\ldots,\tilde{V}_k^{j_p})\tilde{V}_k^{j_{p+1}}
+T^{H,\ldots,H}_{(fu^p)^{[p]}}(\tilde{W},\tilde{W}^{j_2},\ldots,\tilde{W}^{j_p})\tilde{W}^{j_{p+1}}\\ \nonumber
=&\sum_{l=1}^{p+1}T^{H,H+W,H,\ldots,H}_{(fu^p)^{[p]}}(\breve{V}_k,\tilde{V}^{j_2}_k,\ldots,\tilde{V}_k^{j_{l-1}},\tilde{V}_k^{j_l}-\tilde{W}^{j_l},\tilde{W}^{j_{l+1}},\ldots,\tilde{W}^{j_p})\tilde{W}^{j_{p+1}}\\
\nonumber
&-\sum_{l=1}^{p+1}T^{H,\ldots,H}_{(fu^p)^{[p]}}( \tilde{V}_k,\tilde{V}_k^{j_2},\ldots,\tilde{V}_k^{j_{l-1}},\tilde{V}_k^{j_l}-\tilde{W}^{j_l},\tilde{W}^{j_{l+1}},\ldots,\tilde{W}^{j_p})\tilde{W}^{j_{p+1}}\\
\nonumber
=&\sum_{l=1}^{p+1}T^{H,H+W,H,\ldots,H}_{(fu^p)^{[p]}}(\breve{V}_k-\tilde{V}_k,\tilde{V}^{j_2}_k,\ldots,\tilde{V}_k^{j_{l-1}},\tilde{V}_k^{j_l}-\tilde{W}^{j_l},\tilde{W}^{j_{l+1}},\ldots,\tilde{W}^{j_p})\tilde{W}^{j_{p+1}}\\
\nonumber
&+\sum_{l=1}^{p+1}(T^{H,H+W,H,\ldots,H}_{(fu^p)^{[p]}}-T^{H,\ldots,H}_{(fu^p)^{[p]}})( \tilde{V}_k,\tilde{V}_k^{j_2},\ldots,\tilde{V}_k^{j_{l-1}},\tilde{V}_k^{j_l}-\tilde{W}^{j_l},\tilde{W}^{j_{l+1}},\ldots,\tilde{W}^{j_p})\tilde{W}^{j_{p+1}}.
\end{align}
Noticing that $\breve{V}^j-\tilde{V}^j=-\breve{V}^{j+1}$ and applying \eqref{pf} in the last sum in \eqref{rvw2} yields
\begin{align}
\label{rvw2b}
&T^{ H,H+W,H,\ldots,H}_{(fu^p)^{[p]}}(\breve{V}_k,\tilde{V}_k^{j_2},\ldots,\tilde{V}_k^{j_p})\tilde{V}_k^{j_{p+1}}
-T^{ H,H+W,H,\ldots,H}_{(fu^p)^{[p]}}(\breve{W},\tilde{W}^{j_2},\ldots,\tilde{W}^{j_p})\tilde{W}^{j_{p+1}}\\  \nonumber
&\quad-T^{H,\ldots,H}_{(fu^p)^{[p]}}(\tilde{V}_k,\tilde{V}_k^{j_2},\ldots,\tilde{V}_k^{j_p})\tilde{V}_k^{j_{p+1}}
+T^{H,\ldots,H}_{(fu^p)^{[p]}}(\tilde{W},\tilde{W}^{j_2},\ldots,\tilde{W}^{j_p})\tilde{W}^{j_{p+1}}\\
\nonumber
=&\sum_{l=1}^{p+1}\Big(-T^{H,H+W,H,\ldots,H}_{(fu^p)^{[p]}}(\breve{V}_k^2,\tilde{V}^{j_2}_k,\ldots,\tilde{V}_k^{j_{l-1}},\tilde{V}_k^{j_l}-\tilde{W}^{j_l},\tilde{W}^{j_{l+1}},\ldots,\tilde{W}^{j_p})\tilde{W}^{j_{p+1}}\\
\nonumber
&+T^{H,H+W,H,\ldots,H}_{(fu^p)^{[p+1]}}( \tilde{V}_k,W,\tilde{V}_k^{j_2},\ldots,\tilde{V}_k^{j_{l-1}},\tilde{V}_k^{j_l}-\tilde{W}^{j_l},\tilde{W}^{j_{l+1}},\ldots,\tilde{W}^{j_p})\tilde{W}^{j_{p+1}}\Big).
\end{align}
Secondly we handle the summands in \eqref{rvw0} with $j_1\neq1$. By telescoping we obtain
\begin{align}
\label{rvw2c}
&T^{H,\ldots,H}_{(fu^p)^{[p]}}(\breve{V}_k^{j_1},\tilde{V}_k^{j_2},\ldots,\tilde{V}_k^{j_p})\,\tilde{V}_k^{j_{p+1}} -T^{H,\ldots,H}_{(fu^p)^{[p]}}(\breve{W}^{j_1},\tilde{W}^{j_2},\ldots,\tilde{W}^{j_p})\,
\tilde{W}^{j_{p+1}}\\
\nonumber
&-T^{H,\ldots,H}_{(fu^p)^{[p]}}(\tilde{V}_k^{j_1},\tilde{V}_k^{j_2},\ldots,\tilde{V}_k^{j_p})\,\tilde{V}_k^{j_{p+1}} +T^{H,\ldots,H}_{(fu^p)^{[p]}}(\tilde{W}^{j_1},\tilde{W}^{j_2},\ldots,\tilde{W}^{j_p})\,
\tilde{ W}^{j_{p+1}}\\
\nonumber
&=\sum_{l=1}^{p+1}\Big(T^{H,\ldots,H}_{(fu^p)^{[p]}}(\breve{V}_k^{j_1},\tilde{V}^{j_2}_k,\ldots,\tilde{V}_k^{j_{l-1}},\tilde{V}_k^{j_l}-\tilde{W}^{j_l},\tilde{W}^{j_{l+1}},\ldots,\tilde{W}^{j_p})\tilde{W}^{j_{p+1}}\\
\nonumber
&\quad-T^{H,\ldots,H}_{(fu^p)^{[p]}}( \tilde{V}_k^{j_1},\tilde{V}_k^{j_2},\ldots,\tilde{V}_k^{j_{l-1}},\tilde{V}_k^{j_l}-\tilde{W}^{j_l},\tilde{W}^{j_{l+1}},\ldots,\tilde{W}^{j_p})\tilde{W}^{j_{p+1}}\Big)\\
\nonumber
&=-\sum_{l=1}^{p+1}T^{H,\ldots,H}_{(fu^p)^{[p]}}(\breve{V}_k^{j_1+1},\tilde{V}^{j_2}_k,\ldots,\tilde{V}_k^{j_{l-1}},\tilde{V}_k^{j_l}-\tilde{W}^{j_l},\tilde{W}^{j_{l+1}},\ldots,\tilde{W}^{j_p})\tilde{W}^{j_{p+1}}.
\end{align}

Combining \eqref{rvw0}, \eqref{rvw1}, \eqref{rvw2b}, and \eqref{rvw2c} yields
\begin{align}
\label{rvw3a}
&\tilde R^p_{n,H,f}(V_k)-\tilde R^p_{n,H,f}(W)\\
\nonumber
&=\sum_{\substack{j_2,\ldots,j_{p}\geq1\\j_{p+1}\geq0\\j_2+\ldots+j_{p+1}=n-2}}
\Big(T^{H,H+V_k,H+W,H,\ldots,H}_{(fu^p)^{[p+1]}}
(\breve{V}_k,V_k-W,\tilde{V}_k^{j_2},\ldots,\tilde{V}_k^{j_p})\tilde{V}_k^{j_{p+1}}\\
\nonumber
&\quad+\sum_{l=1}^{p+1}\big(T^{H,H+W,H,\ldots,H}_{(fu^p)^{[p+1]}}( \tilde{V}_k,W,\tilde{V}_k^{j_2},\ldots,\tilde{V}_k^{j_{l-1}},\tilde{V}_k^{j_l}-\tilde{W}^{j_l},\tilde{W}^{j_{l+1}},\ldots,\tilde{W}^{j_p})\tilde{W}^{j_{p+1}}\\
\nonumber
&\quad-T^{H,H+W,H,\ldots,H}_{(fu^p)^{[p]}}(\breve{V}_k^2,\tilde{V}^{j_2}_k,\ldots,\tilde{V}_k^{j_{l-1}},\tilde{V}_k^{j_l}-\tilde{W}^{j_l},\tilde{W}^{j_{l+1}},\ldots,\tilde{W}^{j_p})\tilde{W}^{j_{p+1}}\big)\Big)\\
\nonumber
&\quad-\sum_{\substack{j_2,\ldots,j_{p}\geq1\\j_1\ge 2,j_{p+1}\geq0\\j_1+\ldots+j_{p+1}=n-1}}
\sum_{l=1}^{p+1}T^{H,\ldots,H}_{(fu^p)^{[p]}}(\breve{V}_k^{j_1+1},\tilde{V}^{j_2}_k,\ldots,\tilde{V}_k^{j_{l-1}},\tilde{V}_k^{j_l}-\tilde{W}^{j_l},\tilde{W}^{j_{l+1}},\ldots,\tilde{W}^{j_p})\tilde{W}^{j_{p+1}}.
\end{align}
By \eqref{weight1} of Theorem \ref{thm:adding resolvents}\eqref{weight}, for $p\geq1$ we have
\begin{align}
\label{rvw3b}
&T^{H,H+V_k,H+W,H,\ldots,H}_{(fu^p)^{[p+1]}}
(\breve{V}_k,V_k-W,\tilde{V}_k^{j_2},\ldots,\tilde{V}_k^{j_p})\tilde{V}_k^{j_{p+1}}\\
\nonumber
&=T^{H,H+V_k,H+W,H,\ldots,H}_{(fu^{p+1})^{[p+1]}}
(\breve{V}_k,(V_k-W)(H+W-iI)^{-1},\tilde{V}_k^{j_2},\ldots,\tilde{V}_k^{j_p})\tilde{V}_k^{j_{p+1}}\\
\nonumber
&\quad-T^{H,H+V_k,H,\ldots,H}_{(fu^p)^{[p]}}
(\breve{V}_k,(V_k-W)(H+W-iI)^{-1}\tilde{V}_k^{j_2},\tilde{V}_k^{j_3},\ldots,\tilde{V}_k^{j_p})\tilde{V}_k^{j_{p+1}}
\end{align}
and
\begin{align}
\label{rvw3c}
&T^{H,H+W,H,\ldots,H}_{(fu^p)^{[p+1]}}( \tilde{V}_k,W,\tilde{V}_k^{j_2},\ldots,\tilde{V}_k^{j_{l-1}},\tilde{V}_k^{j_l}-\tilde{W}^{j_l},\tilde{W}^{j_{l+1}},\ldots,\tilde{W}^{j_p})\tilde{W}^{j_{p+1}}\\
&\nonumber
=T^{H,H+W,H,\ldots,H}_{(fu^{p+1})^{[p+1]}}( \tilde{V}_k,\tilde{W},\tilde{V}_k^{j_2},\ldots,\tilde{V}_k^{j_{l-1}},\tilde{V}_k^{j_l}-\tilde{W}^{j_l},\tilde{W}^{j_{l+1}},\ldots,\tilde{W}^{j_p})\tilde{W}^{j_{p+1}}\\
\nonumber
&\quad-T^{H,H+W,H,\ldots,H}_{(fu^p)^{[p]}}( \tilde{V}_k,\tilde{W}\tilde{V}_k^{j_2},\tilde{V}_k^{j_3},\ldots,\tilde{V}_k^{j_{l-1}},\tilde{V}_k^{j_l}-\tilde{W}^{j_l},\tilde{W}^{j_{l+1}},\ldots,\tilde{W}^{j_p})\tilde{W}^{j_{p+1}}.
\end{align}
Combining \eqref{rvw3a}--\eqref{rvw3c} yields
\begin{align}
\label{rvw3}
&\tilde R^p_{n,H,f}(V_k)-\tilde R^p_{n,H,f}(W)\\
\nonumber
&=\sum_{\substack{j_1,\ldots,j_{p}\geq1,j_{p+1}\geq0\\
j_1+\ldots+j_{p+1}=n-1}}\bigg(T^{ H,H+V_{k,j_1},H+W_{j_1},H,\ldots,H}_{(fu^{p+1})^{[p+1]}}
(\breve{V}^{j_1}_k,(V_{k,j_1}-W_{j_1})(H+W-iI)^{-1},\tilde{V}^{j_2}_k,\ldots,\tilde{V}^{j_p}_k)\tilde{V}^{j_{p+1}}_k\\
\nonumber
&\quad-T^{H,H+V_{k,j_1},H,\ldots,H}_{(fu^p)^{[p]}}(\breve{V}_k^{j_1},(V_{k,j_1}-W_{j_1})(H+W-iI)^{-1}\tilde{V}_k^{j_2},\ldots,\tilde{V}_k^{j_p})\tilde{V}_k^{j_{p+1}}\\
\nonumber
&\quad+\sum_{l=1}^{p+1}\Big(T^{ H,H+W_{j_1},H,\ldots,H}_{(fu^{p+1})^{[p+1]}}
(\tilde{V}^{j_1}_k,\tilde{W}_{j_1},\tilde{V}_k^{j_2},\ldots,\tilde{V}_k^{j_{l-1}},\tilde{V}_k^{j_l}-\tilde{W}^{j_l},
\tilde{W}^{j_{l+1}},\ldots,\tilde{W}^{j_p})\tilde{W}^{j_{p+1}}\\
\nonumber
&\quad-T^{H,H+W_{j_1},H,\ldots,H}_{(fu^{p})^{[p]}}
(\tilde{V}^{j_1}_k,\tilde W_{j_1} \tilde{V}_k^{j_2},\ldots,\tilde{V}_k^{j_{l-1}},\tilde{V}_k^{j_l}-\tilde{W}^{j_l},
\tilde{W}^{j_{l+1}},\ldots,\tilde{W}^{j_p})\tilde{W}^{j_{p+1}}\\ \nonumber
&\quad-T^{H,H+W_{j_1},H,\ldots,H}_{(fu^{p})^{[p]}}
(\breve{V}^{j_1+1}_k,\tilde{V}_k^{j_2},\ldots,\tilde{V}_k^{j_{l-1}},\tilde{V}_k^{j_l}-\tilde{W}^{j_l},
\tilde{W}^{j_{l+1}},\ldots,\tilde{W}^{j_p})\tilde{W}^{j_{p+1}}\Big)\bigg).
\end{align}

A straightforward application of the second resolvent identity implies
\begin{align*}
\nonumber
(V_k-W)(H+W-iI)^{-1}=(V_k-W)(H-iI)^{-1}(I-W(H+W-iI)^{-1}).
\end{align*}
For each $W\in\{V, V_m\}$, by the estimates \eqref{c=1} of
Lemma \ref{lem:approximating V by V_k}, we obtain
\begin{align}
\label{vfromw}
&\tnrm{\tilde{W}}{n}\leq\tnrm{\tilde{V}}{n}.
\end{align}
and
\begin{align*}
\norm{I-W(H+W-iI)^{-1}}\leq 1+\norm{V}.
\end{align*}
By the latter estimate,
\begin{align*}
\nrm{(V_k-W)(H+W-iI)^{-1}}{n}\leq(1+\norm{V})\|\tilde V_k-\tilde W\|_n.
\end{align*}
It follows from \eqref{vfromw} and the telescoping identity
$\tilde V_k^j-\tilde W^j=\sum_{i=0}^{j-1}\tilde V_k^i(\tilde V_k-\tilde W)\tilde W^{j-1-i}$ that
$$\tnrm{\tilde{V}_k^j-\tilde{W}^j}{n/j}\leq j\tnrm{\tilde{V}}{n}^{j-1}\tnrm{\tilde{V}_k-\tilde{W}}{n}.$$
Applying the latter bound
and Lemma \ref{lem:Hahn Banach Borel measures} in \eqref{rvw3} implies
\begin{align}\label{eq:diff pth remainders}
&|\Tr(\tilde R^p_{n,H,f}(V_k)-\tilde R^p_{n,H,f}(W))|\nonumber\\
&\quad\leq\sum_{\substack{j_1,\ldots,j_{p}\geq1,j_{p+1}\geq0\\j_1+\ldots+j_{p+1}=n-1}}\Big(c^1_{n,j}
\supnorm{(fu^{p+1})^{(p+1)}}+c^2_{n,j}\supnorm{(fu^p)^{(p)}}\Big)C_{n,V,H}
\tnrm{\tilde{V}_k-\tilde{W}}{n},
\end{align}
for some constants $c^1_{n,j}$ and $c^2_{n,j}$ depending only on $n$ and $j_1,\ldots,j_{p+1}$, and the constant $$C_{n,V,H}:=(1+\norm{V})^{2}\,\|\tilde{V}\|_{n}^{n-1}.$$
If $\supp f\subseteq [-a,a]$, then the fundamental theorem of calculus gives
\begin{align*}
\supnorm{(fu^{p})^{(p)}}\leq2a\supnorm{(fu^p)^{(p+1)}}.
\end{align*}
Since $(fu^{p+1})^{(p+1)}=(fu^p)^{(p+1)}u+(p+1)(fu^p)^{(p)}$, we obtain
	$$\supnorm{(fu^{p+1})^{(p+1)}}\leq(|u(a)|+2a(p+1))\supnorm{(fu^p)^{(p+1)}}.$$
Along with \eqref{eq:diff pth remainders}, the latter two inequalities yield the result for $n\geq 3$.

If $n=1$, then $p=0$ and \eqref{R0} gives $\tilde R^0_{1,H,f}(V_k)-\tilde R^0_{1,H,f}(W)=f(H+V_k)-f(H+W)$.
Hence, by Theorem \ref{dm} and the fundamental theorem of calculus,
\begin{align*}
\tilde R^0_{1,H,f}(V_k)-\tilde R^0_{1,H,f}(W)
&=\int_0^1 T^{H_t,H_t}_{f^{[1]}}(V_k-W)\,dt,
\end{align*}
where $H_t=H+W+t(V_k-W)$. By \eqref{weight1} of Theorem \ref{thm:adding resolvents}\eqref{weight} for $j=1$ applied to $T^{H_t,H_t}_{f^{[1]}}(V_k-W)$ and by continuity of the trace, we obtain
\begin{align*}
\Tr(\tilde R^0_{1,H,f}(V_k)-\tilde R^0_{1,H,f}(W))=\int_0^1 \big(&\Tr(T^{H_t,H_t}_{(fu)^{[1]}}((V_k-W)(H_t-iI)^{-1}))\\
&-\Tr(f(H_t)(V_k-W)(H_t-iI)^{-1})\big)\,dt.
\end{align*}
Noticing that
\begin{align}
\label{tvianot}
\sup_{t\in [0,1]}\|(V_k-W)(H_t-iI)^{-1}\|_1&\le(1+\|V_k-W\|)\|\tilde{V}_k-\tilde{W}\|_1\\
\nonumber
&\le(1+2\|V\|)\|\tilde{V}_k-\tilde{W}\|_1
\end{align}
and applying H\"{o}lder's inequality and
the Riesz--Markov representation theorem completes the proof of the result for $n=1$.

If $n=2$, then by Theorem \ref{dm} and the fundamental theorem of calculus,
\begin{align*}
&R_{2,H,f}(V_k)-R_{2,H,f}(W)\\
&=f(H+V_k)-f(H)-T^{H,H}_{f^{[1]}}(V_k)-(f(H+W)-f(H)-T^{H,H}_{f^{[1]}}(W))\\
&=f(H+V_k)-f(H+W)-T^{H,H}_{f^{[1]}}(V_k-W)\\
&=\int_0^1T^{H_t,H_t}_{f^{[1]}}(V_k-W)dt-\int_0^1 T^{H,H}_{f^{[1]}}(V_k-W)dt.
\end{align*}
By \eqref{pf},
\begin{align*}
&R_{2,H,f}(V_k)-R_{2,H,f}(W)\\
&=\int_0^1(T^{H_t,H_t}_{f^{[1]}}(V_k-W)-T^{H,H_t}_{f^{[1]}}(V_k-W)+T^{H,H_t}_{f^{[1]}}(V_k-W)-T^{H,H}_{f^{[1]}}(V_k-W))dt\\
&=\int_0^1(T^{H_t,H,H_t}_{f^{[2]}}(W+t(V_k-W),V_k-W)+T^{H,H_t,H}_{f^{[2]}}(V_k-W,W+t(V_k-W)))dt.
\end{align*}
By Theorem \ref{thm:adding resolvents}\eqref{weights},
\begin{align*}
&T^{H_t,H,H_t}_{f^{[2]}}(W+t(V_k-W),V_k-W)\\
&=f(H_t)(W+t(V_k-W))(H-iI)^{-1}(V_k-W)(H_t-iI)^{-1}\\
&\quad-T_{(fu)^{[1]}}^{H_t,H}\big((W+t(V_k-W))(H-iI)^{-1}\big)(V_k-W)(H_t-iI)^{-1}\\
&\quad-T_{(fu)^{[1]}}^{H_t,H_t}\big((W+t(V_k-W))(H-iI)^{-1}(V_k-W)(H_t-iI)^{-1}\big)\\
&\quad+T_{(fu^2)^{[2]}}^{H_t,H,H_t}\big((W+t(V_k-W))(H-iI)^{-1},(V_k-W)(H_t-iI)^{-1}\big)
\end{align*}
and
\begin{align*}
&T^{H,H_t,H}_{f^{[2]}}(V_k-W,W+t(V_k-W))\\
&=f(H)(V_k-W)(H_t-iI)^{-1}(W+t(V_k-W))(H-iI)^{-1}\\
&\quad-T_{(fu)^{[1]}}^{H,H_t}\big((V_k-W)(H_t-iI)^{-1}\big)(W+t(V_k-W))(H-iI)^{-1}\\
&\quad-T_{(fu)^{[1]}}^{H,H}\big((V_k-W)(H_t-iI)^{-1}(W+t(V_k-W))(H-iI)^{-1}\big)\\
&\quad+T_{(fu^2)^{[2]}}^{H,H_t,H}\big((V_k-W)(H_t-iI)^{-1},(W+t(V_k-W))(H-iI)^{-1}\big).
\end{align*}
Denote
\begin{align*}
R_{t,t,H,W,V_k,f}^0=&f(H_t)(W+t(V_k-W))(H-iI)^{-1}(V_k-W)(H_t-iI)^{-1},\\
R_{t,t,H,W,V_k,f}^1=&-T_{(fu)^{[1]}}^{H_t,H}\big((W+t(V_k-W))(H-iI)^{-1}\big)(V_k-W)(H_t-iI)^{-1}\\
&\quad-T_{(fu)^{[1]}}^{H_t,H_t}\big((W+t(V_k-W))(H-iI)^{-1}(V_k-W)(H_t-iI)^{-1}\big),\\
R_{t,t,H,W,V_k,f}^2=&T_{(fu^2)^{[2]}}^{H_t,H,H_t}\big((W+t(V_k-W))(H-iI)^{-1},(V_k-W)(H_t-iI)^{-1}\big),\\
R_{t,H,W,V_k,f}^0=&f(H)(V_k-W)(H_t-iI)^{-1}(W+t(V_k-W))(H-iI)^{-1},\\
R_{t,H,W,V_k,f}^1=&-T_{(fu)^{[1]}}^{H,H_t}\big((V_k-W)(H_t-iI)^{-1}\big)(W+t(V_k-W))(H-iI)^{-1}\\
&\quad-T_{(fu)^{[1]}}^{H,H}\big((V_k-W)(H_t-iI)^{-1}(W+t(V_k-W))(H-iI)^{-1}\big),\\
R_{t,H,W,V_k,f}^2=&T_{(fu^2)^{[2]}}^{H,H_t,H}\big((V_k-W)(H_t-iI)^{-1},(W+t(V_k-W))(H-iI)^{-1}\big).
\end{align*}
Applying continuity of $t\mapsto\Tr(R_{t,t,H,W,V_k,f}^p)$ and $t\mapsto\Tr(R_{t,H,W,V_k,f}^p)$
(see \cite[Proposition 4.3.15]{ST19}) yields
\begin{align*}
\Tr(R_{2,H,f}(V_k)-R_{2,H,f}(W))=\sum_{p=0}^2\int_0^1\Tr(R_{t,t,H,W,V_k,f}^p+R_{t,H,W,V_k,f}^p)\,dt.
\end{align*}
By Lemma \ref{lem:Hahn Banach Borel measures}, \eqref{c=1}, and an analog of \eqref{tvianot} for the Hilbert-Schmidt norm, we obtain
\begin{align*}
|\Tr(R_{t,t,H,W,V_k,f}^p+R_{t,H,W,V_k,f}^p)|
\leq c_{H,V}\|(fu^p)^{(p)}\|_\infty\|\tilde{V}_k-\tilde{W}\|_2
\end{align*}
for every $t\in[0,1]$, completing the proof of the lemma.
\end{proof}


Below we extend the result of Proposition \ref{prop:SSF locally V_k} to relative Schatten class perturbations.

\begin{proof}[Proof of Proposition \ref{prop:SSF locally}]
Let $(V_k)_k$ be a sequence provided by Lemma \ref{lem:approximating V by V_k}.
For every $p\in\{0,\dots,n-1\}$ and $k\in\N$, let $\breve\eta_{p,k}$ be a function satisfying
$$\Tr(\tilde R^p_{n,H,f}(V_k))=\int(fu^p)^{(p+1)}(x)\breve\eta_{p,k}(x)\,dx,$$
which exists by Proposition \ref{prop:SSF locally V_k}. By Lemma \ref{prop:|Tr(R)|} applied to $W=V_m$, we have
\begin{align*}
\norm{\breve\eta_{p,k}-\breve\eta_{p,m}}_{L^1((-a,a))}
&=\sup_{\substack{ f\in C^{n+1}_c[-a,a]\\
\supnorm{(fu^p)^{(p+1)}}\leq1}}|\Tr(\tilde R^p_{n,H,f}(V_k)-\tilde R^p_{n,H,f}(V_m))|\\
&\leq c_{n,H,V,a}\|\tilde{V}_k-\tilde{V}_m\|_n.
\end{align*}
By Lemma \ref{lem:approximating V by V_k}, the latter expression approaches $0$ as $k\geq m\to\infty$. Thus, $(\breve\eta_{p,k})_k$ is Cauchy in $L^1_\text{loc}(\R)$. Let $\breve\eta_p$ be its $L^1_{\text{loc}}$-limit.

Assume that $f\in C^{n+1}_c[-a,a]$. We obtain
\begin{align*}
\int_\R(fu^p)^{(p+1)}(x)\,\breve\eta_p(x)\,dx
=&\int_{\supp f}(fu^p)^{(p+1)}(x)\,\breve\eta_p(x)\,dx\\
=&\lim_{k\rightarrow\infty}\int_{\supp f}(fu^p)^{(p+1)}(x)\,\breve\eta_{p,k}(x)\,dx\\
=&\lim_{k\rightarrow\infty}\Tr(\tilde R^p_{n,H,f}(V_k)).
\end{align*}
By Lemma \ref{prop:|Tr(R)|} applied to $W=V$,
$$|\Tr (\tilde R^p_{n,H,f}(V_k)-\tilde R^p_{n,H,f}(V))|\leq c_{n,H,V,a}\supnorm{(fu^p)^{(p+1)}}\|\tilde{V}_k-\tilde{V}\|_n\,$$	
for every $k\in\N$.
Hence, by Lemma \ref{lem:approximating V by V_k},
\begin{align*}
\Tr(\tilde R^p_{n,H,f}(V))=\lim_{k\rightarrow\infty}\Tr(\tilde R^p_{n,H,f}(V_k))
=\int_\R(fu^p)^{(p+1)}(x)\,\breve\eta_p(x)\,dx,
\end{align*}
completing the proof of the result.
\end{proof}

\subsection{Absolute continuity of the spectral shift measure}

In this subsection we prove our main result for relative Schatten class perturbations.

\begin{proof}[Proof of Theorem \ref{rsmain}]
Let $f\in C_c^{n+1}$. We provide a proof in the case $n\ge 3$; the cases $n=1$ and $n=2$ can be proved completely analogously.

Applying the general Leibniz differentiation rule on the right hand side of \eqref{trfp} (see Proposition \ref{prop:SSF locally}) gives
\begin{align*}
\Tr(R_{n,H,f}(V))
&=\sum_{p=0}^{n-1}(-1)^{n-1-p}\int_\R(fu^p)^{(p+1)}(x)\,\breve\eta_p(x)\,dx.\\
&=\sum_{p=0}^{n-1}(-1)^{n-1-p}\sum_{k=0}^{p+1}\int_\R\vect{p+1}{k} f^{(k)}(x)(u^p)^{(p+1-k)}(x)\breve\eta_p(x)\,dx\,\\
&=\sum_{p=0}^{n-1}(-1)^{n-1-p}\sum_{k=1}^{p+1}\int_\R f^{(k)}(x) \vect{p+1}{k}\frac{p!}{(k-1)!}u^{k-1}(x)\breve\eta_p(x)\,dx.
\end{align*}
Integration by parts gives
\begin{align}
\nonumber
\Tr(R_{n,H,f}(V))=&\sum_{p=0}^{n-1}\int_\R f^{(p+1)}(x)\tilde{\eta}_p(x)\,dx,
\end{align}
where
\begin{align*}
\tilde{\eta}_p(t)
=\sum_{k=1}^{p+1}\frac{(-1)^{n-k}\,(p+1)!\,p!}{(p+1-k)!\,k!\,(k-1)!}\int_0^t ds_1\int_0^{s_1}ds_2\cdots\int_0^{s_{p-k}}u^{k-1}(x)\breve\eta_p(x)\,dx.
\end{align*}
Subsequent integration by parts gives
\begin{align}
\label{eta}
\nonumber&\Tr(R_{n,H,f}(V))\\
\nonumber
&=\int_\R f^{(n)}(x)\Bigg(\sum_{p=0}^{n-1}(-1)^{n-1-p}\int_0^xds_1\int_0^{s_1}ds_2\cdots
\int_0^{s_{n-p-2}}\tilde{\eta}_p(t)\,dt\Bigg)\,dx\quad\\
 &=:\int_\R f^{(n)}(x)\grave\eta_n(x)\,dx
\end{align}
for every $f\in C_c^{n+1}$.
Since $\breve{\eta}_p\in L^1_{\text{loc}}$ (see Proposition \ref{prop:SSF locally}),
we have that $\tilde{\eta}_p\in L^1_{\text{loc}}$ and, hence, $\grave\eta_n\in L^1_{\text{loc}}$.

By Proposition \ref{prop:SSM with growth}, there exists a locally finite Borel measure $\mu_n$ satisfying \eqref{mu tilde} and determined by \eqref{mu tilde} for every $f\in C_c^{n+1}$ uniquely up to an absolutely continuous measure whose density is a polynomial of degree at most $n-1$.
Combining the latter with \eqref{eta} implies
\begin{align}
\label{311}
d\mu_n(x)=\grave\eta_n(x)dx+p_{n-1}(x)dx =:\acute\eta_n(x)dx,
\end{align}
where $p_{n-1}$ is a polynomial of degree at most $n-1$. By Proposition \ref{prop:SSM with growth}, the function $\acute\eta_n:=\grave\eta_n+p_{n-1}$ satisfies \eqref{tff} for every $f\in\mW_n$. The fact that $u^{-n-\epsilon}d\mu_n$ is a finite measure translates to $\acute\eta_n\in L^1(\R,u^{-n-\epsilon}(x)dx)$.

It follows from \eqref{munfla} that
\begin{align*}
\|u^{-n-\epsilon}\,d\mu_n\|\le\|u^{-\epsilon}\|_\infty\|\nu_n\|
+\|u^{-n-\epsilon}\,\xi_n\|_1.
\end{align*}
Along with \eqref{eq:nu bound} and \eqref{eq:xi bound}, the latter implies
\begin{align*}
\|u^{-n-\epsilon}\,d\mu_n\|\le c_n(1+\|u^{-1-\epsilon}\|_1)(1+\norm{V})\nrm{V(H-iI)^{-1}}{n}^n.
\end{align*}
Since
\begin{align}
\label{312}
\int_0^1(1+x^2)^{(-1-\epsilon)/2}\,dx\le 1\quad\text{and}\quad
\int_1^\infty(1+x^2)^{(-1-\epsilon)/2}\,dx\le\int_1^\infty x^{-1-\epsilon}\,dx=\epsilon^{-1},
\end{align}
we obtain the bound
\begin{align}
\label{313}
\|u^{-n-\epsilon}\,d\mu_n\|\le c_n\,(1+\epsilon^{-1})(1+\norm{V})\nrm{V(H-iI)^{-1}}{n}^n,
\end{align}
which translates to
	$$\int_\R |\acute\eta_n(x)|\,\frac{dx}{(1+|x|)^{n+\epsilon}}\leq c_n(1+\epsilon^{-1})(1+\norm{V})\|V(H-iI)^{-1}\|_n^n.$$
We define $$\eta_n:=\Re(\acute\eta_n),$$ and obtain \eqref{eta estimate} by using $|\eta_n|\leq|\acute\eta_n|$.
As we have seen, $\acute\eta_n$ satisfies \eqref{tff} for all $f\in\mW_n$. Therefore,
\begin{align}
\label{eq:tr formula real and im part}
	\Tr(R_{n,H,f}(V))=\int_\R f^{(n)}(x)\eta_n(x)\,dx+i\int_\R f^{(n)}(x)\Im(\acute\eta_n(x))\,dx.
\end{align}
When $f\in \mW_n$ is real-valued, the left hand side of \eqref{eq:tr formula real and im part} is real, and consequently the second term on the right hand side of \eqref{eq:tr formula real and im part} vanishes. The latter implies \eqref{tff} for real-valued $f\in\mW_n$. By applying \eqref{tff} to the real-valued functions $\Re(f)$ and $\Im(f)$, we extend \eqref{tff} to all $f\in\mW_n$.

The uniqueness of $\eta_n$ satisfying \eqref{tff} up to a polynomial summand of order at most $n-1$ can be established completely analogously to the uniqueness of the measure $\mu_n$ established in Proposition \ref{prop:SSM with growth}.
\end{proof}

\section{Examples}
\label{sec4}

In this section we discuss models of noncommutative geometry and mathematical physics that satisfy the condition \eqref{vresinsn}.

\subsection{Noncommutative geometry}
\label{sec4a}
In this subsection we show that the relative Schatten class condition occurs naturally in noncommutative geometry, namely, in inner perturbations of regular locally compact spectral triples (see Definition \ref{lcst} below). Many examples, including noncommutative field theory \cite{GGISV}, satisfy the following definition.

Let $\operatorname{dom}(D)$ denote the domain of any operator $D$
and let \begin{align*}
\delta_D(T):=[|D|,T]
\end{align*}
be defined on those $T\in\mB(\H)$ for which $\delta_D(T)$ extends to a bounded operator.

\begin{defn}
\label{lcst}
A locally compact spectral triple $(\mathcal{A},\H,D)$ consists of a separable Hilbert space $\H$, a self-adjoint operator $D$ in $\H$ and a *-algebra $\mathcal{A}\subseteq B(\H)$ such that $a(\operatorname{dom}(D))\subseteq\operatorname{dom}(D)$, $[D,a]$ extends to a bounded operator,
and $a(D-iI)^{-s}\in \S^1$ for all $a\in\mathcal{A}$ and some $s\in\N$, called the summability of $(\mathcal{A},\H,D)$.
A spectral triple $(\mathcal{A},\H,D)$ is called regular if for all $a\in\mathcal{A}$, we have $a,[D,a]\in\bigcap_{k=1}^\infty \dom(\delta_D^k)$.
\end{defn}

The following result appears to be known, but nowhere explicitly proven, although a similar statement is made in \cite{SZ18}.

Let $\Omega^1_D(\mathcal{A}):=\{\sum_{j=1}^n a_j[D,b_j]: a_j,b_j\in\mathcal{A}, n\in\N\}$ denote the set of inner fluctuations \cite{CC97} or \textit{Connes' differential one-forms}.

\begin{thm}
	A regular locally compact spectral triple $(\mathcal{A},\H, D)$ of summability $s$ satisfies $V(D-iI)^{-1}\in\S^s$ for all $V\in\Omega^1_D(\mathcal{A})$.
\end{thm}

\begin{proof}
Let $V=\sum_{j=1}^n a_j[D,b_j]\in\Omega^1_D(\mathcal{A})$ be arbitrary and let $\delta:=\delta_D$.
For all $X\in\bigcap_{k=1}^\infty \dom(\delta^k)$ we have
\begin{align*}
X(|D|-iI)^{-1}=(|D|-iI)^{-1}X+(|D|-iI)^{-1}\delta(X)(|D|-iI)^{-1}.
\end{align*}
By induction, for all $X\in\bigcap_{k=1}^\infty\dom(\delta^k)$ there exists some $Y\in\bigcap_{k=1}^\infty\dom(\delta^k)$ such that
\begin{align}\label{eq:Commute with powers of resolvents of |D|}
X(|D|-iI)^{-s}=(|D|-iI)^{-s}Y.
\end{align}

Since $[D,b_j]\in \bigcap_{k=1}^\infty \dom(\delta^k)$ for all $j$ and since $g:\R\to\C,t\mapsto (|t|-i)/(t-i)$ is continuous and bounded, we have $g(D)\in\mB(\H)$ and there exists some $Y_j\in\mB(\H)$ such that
\begin{align*}
	V(D-iI)^{-s}&=\sum_j a_j[D,b_j](|D|-iI)^{-s}g(D)^s\\
	&=\sum_j a_j(|D|-iI)^{-s}Y_jg(D)^s=\sum_j a_j(D-iI)^{-s}g(D)^{-s}Y_jg(D)^s\in\S^1.
\end{align*}

More generally, let $X_1,\ldots,X_m\in\bigcap_{k=1}^\infty \dom(\delta^k)$, let $k_1,\ldots,k_m\in\N$ and set $k=\sum_{j=1}^m k_j$. By induction, noting that $\bigcap_{k=1}^\infty\dom(\delta^k)$ is an algebra, and applying \eqref{eq:Commute with powers of resolvents of |D|} to $s=k_j$, we obtain
	\begin{align*}
		\prod_{j=1}^m X_j(D-iI)^{-k_j}=(D-iI)^{-k}Y,
	\end{align*}
	for some $Y\in\bigcap_{k=1}^\infty \dom(\delta^k)$. If $s$ is even, we obtain
	\begin{align*}
		|(D+iI)^{-1}V^*|^{s}&=V(D^2+I)^{-1}V^*\cdots V(D^2+I)^{-1}V^*\\
		&=V(D-iI)^{-s}Y\in\S^1,
	\end{align*}
	for some $Y\in\bigcap_{k=1}^\infty\dom(\delta^k)$.
Therefore, $V(D-iI)^{-1}=((D+iI)^{-1}V^*)^*\in\S^{s}$.

If $s$ is odd, we use polar decomposition to obtain $U\in\mB(\H)$ such that $|V(D-iI)^{-1}|=UV(D-iI)^{-1}$. Hence,
\begin{align*}
		|V(D-iI)^{-1}|^{s}&=UV(D-iI)^{-1}|V(D-iI)^{-1}|^{s-1}\\
		&=UV(D^2+I)^{-1}V^*\cdots V(D^2+I)^{-1}V^*V(D-iI)^{-1}\\
		&=UV(D-iI)^{-s}Y'\in\S^1
	\end{align*}
	for some $Y'\in\bigcap_{k=1}^\infty\dom(\delta^k)$. Therefore, $V(D-iI)^{-1}\in\S^s$.
\end{proof}

\subsection{Differential operators}
\label{sec4b}
In this section we consider conditions sufficient for perturbations of Dirac and Schr\"{o}dinger operators to satisfy \eqref{vresinsn}.

Given $v\in L^\infty(\R^d)$, let $M_v$ denote the operator of multiplication by $v$, that is,
$$M_v(g):=vg,\quad g\in L^2(\R).$$
We will consider self-adjoint perturbations $V=M_v$, where $v$ is real-valued.

Let $$\Delta=\sum_{k=1}^d\frac{\partial^2}{\partial x_k^2}$$ denote the Laplacian operator densely defined in the Hilbert space $ L^2(\mathbb{R}^d)$.

For $m\geq0$, let $D_m$ denote the free massive Dirac operator defined as follows. For $d\in\mathbb{N},$ let $N(d):=2^{\lfloor (d+1)/2\rfloor}$. Let $e_k\in M_{N(d)}(\mathbb{C}),$ $0\leq k\leq d,$ be the Clifford generators, that is, self-adjoint matrices satisfying $e_k^2=I$ for $0\leq k\leq d$ and
$e_{k_1}e_{k_2}=-e_{k_2}e_{k_1}$ for $0\leq k_1,k_2\leq d,$ such that
$k_1\neq k_2.$ Let $D_k:=\frac{\partial}{i\partial x_k}$. Then, the operator
$$D_m:=e_0\otimes mI+\sum_{k=1}^de_k\otimes D_k$$
is densely defined in the Hilbert space $\mathbb{C}^{N(d)}\otimes L^2(\mathbb{R}^d).$

We note that $D_0$ is unitarily equivalent to $I\otimes D$, where $I\in M_{N(d)/N(d-1)}(\C)$ and $D$ is the usual massless Dirac operator. We also note that, in the case when $d=1$, the Dirac operator $D_0=I\otimes\frac{\partial}{i\partial x}$ can be identified with the differential operator $\frac{\partial}{i\partial x}$ in the Hilbert space $L^2(\R)$.

The Schatten class membership of the weighted resolvents below was derived in \cite[Theorem 3.3 and Remark 3.6]{S21}.
To estimate the respective Schatten norms one just needs to carefully follow the proof of the latter result.

\begin{thm}
\label{src}
Let $d\in\N$, $1\le p<\infty$. Let
$$v\in\begin{cases}\ell^p(L^2(\R^d))\cap L^\infty(\R^d) &\text{if }\; 1\le p<2\\
L^p(\R^d)\cap L^\infty(\R^d) &\text{if }\; 2\le p<\infty\end{cases}$$
be real-valued.
\begin{enumerate}[(i)]
\item \label{srci}
If $p>d$ and $m\ge 0$, then
$(I\otimes M_v)(D_m-iI)^{-1}\in\S^p$
and
\begin{align}
\label{wrd}
\|(I\otimes M_v)(D_m-iI)^{-1}\|_p\le c_{d,p}\begin{cases}\|v\|_{\ell^p(L^2)} &\text{if }\; 1\le p<2\\
\|v\|_{L^p} &\text{if }\; 2\le p<\infty.\end{cases}
\end{align}

\item \label{srcii}
If $p>\frac{d}{2}$, then
$M_v(-\Delta-iI)^{-1}\in\S^p$
and
\begin{align}
\label{wrs}
\|M_v(-\Delta-iI)^{-1}\|_p\le c_{d,p}\begin{cases}\|v\|_{\ell^p(L^2)} &\text{if }\; 1\le p<2\\
\|v\|_{L^p} &\text{if }\; 2\le p<\infty.\end{cases}
\end{align}
\end{enumerate}
\end{thm}

\begin{remark}
\label{perturbedr}
The bounds analogous to \eqref{wrd} and \eqref{wrs} can also be established for perturbed Dirac $D_m+W$ and perturbed Schr\"{o}dinger $-\Delta+W$ operators, respectively. The respective results follow from Theorem \ref{src} and
Proposition \ref{perturbedp} below. In particular, we have the following bound for a massive Dirac operator with
electromagnetic potential in the case $p>d$:
\begin{align*}
&\Big\|(I\otimes M_v)\Big(D_m+\sum_{k=1}^{d}e_k\otimes M_{w_k}+I\otimes M_{w_{d+1}}-iI\Big)^{-1}\Big\|_p\\
&\quad\le c_{d,p}\big(1+\max_{1\le k\le d+1}\|w_k\|_{L^\infty}\big)\begin{cases}\|v\|_{\ell^p(L^2)} &\text{if }\; 1\le p<2\\
\|v\|_{L^p} &\text{if }\; 2\le p<\infty.\end{cases}
\end{align*}
The same reasoning applies to generalized Dirac operators $I_k\otimes D+W$, where $k\in\N$ and $W\in\mB(\C^k\otimes\H)_{\text{sa}}$, that are associated with almost-commutative spectral triples.
\end{remark}

\begin{prop}
\label{perturbedp}
Let $H,V$ be self-adjoint operators in $\H$ and $W\in\mB(\H)_{\text{sa}}$. Let $1\le p<\infty$ and assume that
$\|V(H-iI)^{-1}\|_p<\infty$. Then,
\begin{align*}
\|V(H+W-iI)^{-1}\|_p\le\|V(H-iI)^{-1}\|_p(1+\|W\|).
\end{align*}
\end{prop}

\begin{proof}
The result follows from the second resolvent identity
\begin{align*}
(H+W-iI)^{-1}=(H-iI)^{-1}-(H-iI)^{-1}W(H+W-iI)^{-1}
\end{align*}
upon multiplying it by $V$ and applying H\"{o}lder's inequality for Schatten norms.
\end{proof}


\end{document}